\documentclass{article}
\usepackage[final]{nips_2016}
\usepackage{times}
\usepackage{graphicx} 
\usepackage{subfigure} 

\usepackage{natbib}

\usepackage{algorithm}
 \usepackage{algorithmic}

\usepackage{amsmath,amssymb}
\usepackage{footnote}
\makesavenoteenv{tabular}
\makesavenoteenv{table}
\usepackage{mathrsfs}
\usepackage{xcolor}
\usepackage{enumerate}
\usepackage{enumitem}

\usepackage[utf8]{inputenc} 
\usepackage[T1]{fontenc}    
\usepackage{url}            
\usepackage{booktabs}       
\usepackage{amsfonts}       
\usepackage{nicefrac}       
\usepackage{microtype}      
\usepackage{undertilde}

\newtheorem{remark}{Remark}

\newtheorem{theorem}{Theorem}
\newtheorem{lemma}{Lemma}

\newtheorem{proposition}[theorem]{Proposition}

\def \endprf{\hfill {\vrule height6pt width6pt depth0pt}\medskip}
\newenvironment{proof}{\noindent {\bf Proof} }{\endprf\par}

\newcommand{\OurMethod}{HONES~}
\newcommand{\com}[2]{\left(\begin{array}{c}#1 \\ #2\end{array}\right)}
\newcommand{\lb}{\left(}
\newcommand{\rb}{\right)}
\newcommand{\eps}{\epsilon}

\newcommand{\td}{\tilde}

\newcommand{\R}{\mathbb{R}}

\newcommand{\utd}[1]{\utilde{#1}}
\newcommand{\g}{g}
\newcommand{\gs}{\g_{S}}
\newcommand{\gc}{\g_{S^{c}}}
\newcommand{\y}{y}
\newcommand{\ys}{\y_{S}}
\newcommand{\yc}{\y_{S^{c}}}
\newcommand{\one}{\textbf{1}}
\newcommand{\ones}{\one_{S}}
\newcommand{\onec}{\one_{S^{c}}}
\newcommand{\A}{A}
\newcommand{\Ass}{\A_{SS}}
\newcommand{\Asc}{\A_{SS^{c}}}
\newcommand{\Acs}{\A_{S^{c}S}}
\newcommand{\Acc}{\A_{S^{c}S^{c}}}
\newcommand{\x}{x}
\newcommand{\xs}{x_{S}}
\newcommand{\xc}{x_{S^{c}}}
\newcommand{\mus}{\mu_{S}}
\newcommand{\muc}{\mu_{S^{c}}}
\newcommand{\mat}[4]{\lb
  \begin{array}{cc}
    #1 & #2\\
    #3 & #4
  \end{array}
\rb}
\renewcommand{\c}{r}
\renewcommand{\S}{S}
\newcommand{\cs}{\c_{\S}}
\newcommand{\cc}{\c_{\S^{c}}}
\newcommand{\teta}{\td{\eta}}
\newcommand{\lam}{\lambda}
\newcommand{\flam}{\alpha(\lambda)}
\newcommand{\D}{D}
\newcommand{\Dg}{D_{g}}
\newcommand{\Dgg}{D_{gg}}
\newcommand{\Dgy}{D_{g\c}}

\newcommand{\M}{M}
\newcommand{\Mc}{\M_{\bigcdot,\S^{c}}}
\newcommand{\Ms}{\M_{\bigcdot,\S}}
\newcommand{\Mss}{\M_{\S\S}}

\newcommand{\Mcs}{\M_{\S^{c}\S}}

\newcommand{\etas}{\eta_{\S}}
\newcommand{\etac}{\eta_{\S^{c}}}
\newcommand{\tetas}{\teta_{\S}}
\newcommand{\tetac}{\teta_{\S^{c}}}
\newcommand{\use}[1]{&&[\mbox{\emph{Use }}#1]}
\newcommand{\Par}{\mathrm{Par}}
\renewcommand{\v}{v}
\newcommand{\T}{\td{\S}}
\newcommand{\tAjj}{\td{\A}_{jj}}
\newcommand{\Ajj}{\A_{jj}}
\newcommand{\Asj}{\A_{\S j}}
\newcommand{\Ajs}{\A_{j\S}}

\newcommand{\Atj}{\A_{\T j}}
\newcommand{\Ajt}{\A_{j\T}}

\newcommand{\Mt}{\M_{\bigcdot, \T}}
\newcommand{\MC}{\M_{\bigcdot, \T^{c}}}
\newcommand{\Mjs}{\M_{j\S}}

\newcommand{\Mtt}{\M_{\T\T}}
\newcommand{\Mtj}{\M_{\T j}}
\newcommand{\Mjt}{\M_{j\T}}
\newcommand{\MCt}{\M_{\T^{c}\T}}

\newcommand{\MCs}{\M_{\T^{c}\S}}

\newcommand{\MCj}{\M_{\T^{c}j}}

\newcommand{\Mjj}{\M_{jj}}
\newcommand{\Mct}{\M_{\S^{c}\T}}
\newcommand{\Mcj}{\M_{\S^{c}j}}

\newcommand{\Att}{\A_{\T\T}}
\newcommand{\ACj}{\A_{\T^{c}j}}

\newcommand{\Acj}{\A_{\S^{c}j}}
\newcommand{\ACt}{\A_{\T^{c}\T}}

\newcommand{\Act}{\A_{\S^{c}\T}}
\newcommand{\ACs}{\A_{\T^{c}\S}}

\newcommand{\rcom}[2]{(#1 \quad #2)}
\newcommand{\opr}{\mathcal{R}_{j}}

\newcommand{\gC}{\g_{\T^{c}}}
\newcommand{\gt}{\g_{\T}}

\newcommand{\onet}{\one_{\T}}

\newcommand{\ulam}{\utd{\lambda}}
\newcommand{\h}{h}

\renewcommand{\l}{\ell}
\newcommand{\ls}{\l_{\S}}
\newcommand{\lc}{\l_{\S^{c}}}
\newcommand{\lt}{\l_{\T}}

\newcommand{\xis}{\xi_{\S}}
\newcommand{\xic}{\xi_{\S^{c}}}
\newcommand{\xit}{\xi_{\T}}

\newcommand{\del}{\delta}

\newcommand{\Dh}{\D_{h}}
\newcommand{\Dl}{\D_{\l}}

\newcommand*{\bigcdot}{\raisebox{-0.25ex}{\scalebox{1.2}{$\cdot$}}}
\DeclareMathOperator*{\supp}{supp}

\title{HONES: A Fast and Tuning-free Homotopy Method For Online Newton Step}

\author{
  Yuting Ye \\
  Division of Biostatistics \\
  Univ. of California, Berkeley \\
  \texttt{yeyt@berkeley.edu} \\
  \And
  Cheng Ju \\
  Division of Biostatistics\\
  Univ. of California, Berkeley\\
  \texttt{cju@berkeley.edu} \\
  \And
  Lihua Lei \\
  Department of Statistics \\
  Univ. of California, Berkeley \\
  \texttt{lihua.lei@berkeley.edu} \\
}

\begin{document}
\maketitle
%

%

\begin{abstract}
In this article, we develop and analyze a homotopy continuation method, referred to as \OurMethod, for solving the sequential generalized projections in Online Newton Step \citep{hazan2006logarithmic}, as well as the generalized problem known as \emph{sequential standard quadratic programming}. \OurMethod is fast, tuning-free, error-free (up to machine error) and adaptive to the solution sparsity. This is confirmed by both careful theoretical analysis and extensive experiments on both synthetic and real data.
\end{abstract}

\section{Introduction}\label{sec:intro}
Online convex optimization (OCO) is an appealing framework that unifies online and sequential optimization problems in various areas. In OCO, a player sequentially makes decisions by choosing a point in a convex set and a concave payoff function is revealed after each decision. The player aims to ``maximize'' her cumulative payoff, or formally minimize the \emph{regret}, which measures the gap between the average payoff of her decision strategy and that of the best fixed-action strategy from hindsight. One of the high-profile motivation is the universal portfolio management problem \citep{cover1991universal}, where an investor seeks an online strategy to allocate her wealth on a set of financial instruments without making any assumption on the market behaviors. The payoff can be quantified by \emph{logarithmic wealth growth ratio}, formulated as $\sum_{t=1}^{T}\log (x_{t}^{T}\gamma_{t})$ where $x_{t}(j) \,\,(j = 1,\ldots, n)$ is the share of the $j$-th stock in the portfolio and $\gamma_{t}(j)$ is the ratio of the closing price of stock $j$ on time $t$ to that on time $t - 1$. In view of the prohibition of short sales in most markets, the decision space is thus the $n$-dimensional simplex $\Delta_{n} = \left\{x\in \R^{n}: \sum_{i=1}^{n}x_{i} = 1, x_{i}\ge 0\right\}$. The regret of a given strategy that outputs $\{x_{t}\}_{t=1}^{T}$ is then
\begin{equation}
  \label{eq:regret}
  \sup_{x\in \Delta_{n}}\sum_{t=1}^{T}\log (x^{T}\gamma_{t}) - \sum_{t=1}^{T}\log (x_{t}^{T}\gamma_{t}).
\end{equation}

A rich class of algorithms has been developed since \cite{cover1991universal} which proposed an algorithm with regret $O(\sqrt{n\log T})$ but with exponential computation cost per period. \cite{helmbold98} developed an algorithm that reduces the computation cost to $O(n)$ but incurs a sub-optimal regret $O(\sqrt{T\log n})$ in terms of horizon dependence. Later \cite{kalai02} gave an polynomial-time algorithm that achieves the $O(\sqrt{n}\log T)$ regret, though the order of polynomials is still high. In 2003, the pioneering work by \cite{zinkevich03} proposed the influential \emph{Online Gradient Method} which achieves $O(\sqrt{T})$ regret for general OCO problems. The next milestone, among others, is achieved by \cite{hazan2006logarithmic}, which proposed \emph{Online Newton Step} that achieves $O(\log T)$ regret under mild conditions, satisfied in universal portfolio management problems, with a practical computation cost per period. \cite{hazan2006logarithmic} also shows that Online Gradient Method is able to achieve $O(\log T)$ regret but requires the loss functions to be strongly convex and, more stringently, the player knowing the strong-convexity modulus apriori. We refer the readers to \cite{shalev2011online} for the history and to Elad Hazan's thesis \citep{hazan06} for detailed description of Online Newton Step.

Despite the promising theoretical guarantee of Online Newton Step, the computation efficiency remains a considerable concern for practitioners as the algorithm involves solving a sequence of \emph{generalized projections}. Specifically, at time $t$ one needs to solve
\begin{align}
\min \frac{1}{2}x^{T}\A^{(t)} x - (\c^{(t)})^{T}x, \quad \ensuremath{s.t.} \quad x\in \Delta_{n}\label{eq:stqp}
\end{align}
where $A^{(t)}$ (resp. $\c^{(t)}$) is a sequence of matrices (resp. vectors) such that
\begin{equation}\label{eq:update}
A^{(t + 1)} = A^{(t)} + g^{(t)} (g^{(t)})^{T},
\end{equation}
for some time-varying vectors $g^{(t)}$. In the special case $A = I$, \eqref{eq:stqp} can be solved quite efficiently in $O(n)$ time \citep{duchi2008efficient} due to the explicit form of the solution. Unfortunately, in Online Newton Step, $A^{(t)}$ is never a scaled identity matrix and such benefits disappear for general matrices. \cite{hazan2006logarithmic} suggests using iterative algorithms such as interior-point method \citep{wright1997primal}. However, it is known that interior-point method has $O(n^{3})$ computation cost per iteration, which could be prohibitive for large problems or high-frequency online problems. For this reason, the sub-problem \eqref{eq:stqp} becomes the bottleneck of Online Newton Step, which motivates our work. 

Interestingly, \eqref{eq:stqp} is also the generic problem in other areas such as Markowitz's portfolio management \citep{markowitz1952portfolio} and resource allocation \citep{ibaraki88}. This is referred to as \emph{standard quadratic optimization} dating back to 1950s; See \cite{bomze98}.

Without the rank-one update structure \eqref{eq:update}, we should not expect significant improvement over interior-point method as \eqref{eq:stqp} leads to multiple unrelated quadratic programming problems. Nevertheless, \ref{eq:update} is a ``huge bonus'' that connects the consecutive problems: In fact $A^{(t)}$ is perturbed in only one direction at each step and hence the optimal solutions in consecutive steps should be close. A widely used strategy to exploit the minor change is warm-start, i.e. initializing the iterate as the optimal solution in the last step. Spectral projected gradient (SPG) method \citep{birgin2000nonmonotone} is a typical algorithm falling into this category, which combines projected gradient method with smart line search. However, a warm-start is not always allowed. For example, the interior-point method \citep{wright1997primal} requires the initializer to be an interior point of the constraint set, but as shown in various settings and applications, including our experiments in section \ref{sec:experiments}, the solution in each step often lies on the boundary of the simplex. Another potential algorithm is Exponentiated Gradient Descent \citep{kivinen97} or Mirror Descent \citep{beck03}. However, it also requires the initializer to be an interior point in that any zero entry will stay zero. Furthermore, it is lack of an efficient stopping rule, which might not be essential for solving a single problem but is quite important for solving thousands of problems.

On the other hand, it has been proved that the minimizer of a standard quadratic programming problem tends to be sparse under fairly general structural assumptions \citep{chen13, chen2015new}. We also observed the sparsity in both synthetic and real datasets; see section \ref{sec:experiments} for details. However, none of existing algorithm takes the solution sparsity into account. \footnote{We hope the readers not be confused by the word "sparsity", which usually appears as an ``assumption'' on underlying parameters in literature. Here it is a ``phenomenon'', observed in both theory and practice, that the solution of the optimization problem tends to be sparse.}

In summary, to the best of our knowledge, existing methods are neither tailored for the sequential problem with structure \eqref{eq:update} nor designed to adapt to the solution sparsity. To exploit the structure \eqref{eq:update}, we resort to the homotopy method, which is proposed decades ago and widely used in optimizing highly non-convex problems such as polynomial systems \citep{chow79, li83}. The basic idea is to construct a bivariate function $H(x, w)$ on $\R^{n}\times [0, 1]$ with $H(x, 0) = g(x)$ and $H(x, 1) = f(x)$. In order to optimize $f(x)$ one can start from the optimizer of $g(x)$ and move towards $f(x)$ by gradually increasing $w$. Given sufficient smoothness of $H$ along with the non-singularity of the Hessian matrix of $H$ w.r.t $x$, one can obtain a smooth trajectory, or a solution path , penetrating the optimizers of $H(x, w)$ for all $w\in [0, 1]$ with the optimizer of $f(x)$ being the ending point. Homotopy methods for quadratic programming problems have been studied and applied for decades \citep{frank56, banknon, ritter1981parametric, murty1988linear, best1996algorithm, efron04}. However, all these methods are designed for a specific problem. Recently, the homotopy methods have been applied to sequential problems. For example, \citet{garrigues09} proposed a homotopy method to solve the online LASSO regression problem where the objectives are updated in a similar fashion as \eqref{eq:update}.

For our problem \eqref{eq:stqp}, we define the homotopy function in a zigzag fashion which moves $(\A^{(t-1)}, \c^{(t-1)})$ to $(\A^{(t)}, \c^{(t-1)})$ and to $(\A^{(t)}, \c^{(t)})$ then; See Section \ref{subsubsec:construct} for the explicit construction. We show that the solution path can be calculated efficiently and \emph{exactly} using the Karush-Kuhn-Tucker (KKT) conditions. By carefully analyzing the evolution of solutions, we propose an algorithm, referred to as \emph{Homotopy Online NEwton Step (HONES)}, which is \emph{fast}, \emph{tuning-free}, \emph{error-free} (up to machine error) and \emph{adaptive to solution sparsity}.

We compute the number of atomic operations exactly, up to an additive constant, in Theorem \ref{thm:complexity}. As almost all other homotopy continuation methods, the theoretical complexity of our algorithm is in general incomparable to other iterative algorithms like SPG and interior-point method because the former is proportional to the number of turning points on the trajectory (see Section \ref{subsec:turningpoints}) while the latter is proportional to the number of iterations to achieve an accurate solution (see Section \ref{subsec:complexity}). For this reason, we compare the algorithms by the running time on both synthetic and real datasets. To conclude, \OurMethod has a superior performance to SPG and interior-point method and the gain of computational efficiency of \OurMethod is more significant when the solutions are sparser.

The rest of the paper is organized as follows: \OurMethod algorithm is detailed in Section \ref{sec:algo}, followed by the theory and complexity analysis. The practical implementation is more delicate than the general idea and hence stated in the Supplementary Material. In Section \ref{sec:experiments}, we apply \OurMethod to NYSE and NASDAQ data using Online Newton Step for universal portfolio management. We also conduct experiments on synthetic data and for Markowitz's portfolio management on real data. Section \ref{sec:conclusion} concludes the article.

\section{Proposed Algorithm}\label{sec:algo}
A generic framework to solve problem \eqref{eq:stqp} with matrix flow \eqref{eq:update} is summarized in Algorithm \ref{algo:generic} where ALGO1 and ALGO2 could be arbitrary sub-routines producing the solution of \eqref{eq:stqp} in step $0$ and the following steps.

\begin{algorithm}
\caption{Framework to solve (\ref{eq:stqp})}\label{algo:generic}
\textbf{Inputs: } Initial matrix $\A^{(0)}$, vectors $\{\g^{(t)}, \c^{(t)}, t = 1, 2, \ldots\}$.

\textbf{Procedure: }
\begin{algorithmic}[1]
  \STATE Initialize: $x^{(0)}\gets \mathrm{ALGO1}(\A^{(0)}, \c^{(0)})$;
  \FOR{$t = 1, 2, \cdots$.}
  \STATE $x^{(t)}\gets \mathrm{ALGO2}(\A^{(t - 1)}, g^{(t)}, \c^{(t)}; x^{(t - 1)})$;
  \ENDFOR
\end{algorithmic}\label{algo:outer_loop}

\textbf{Output: }$\{x^{(t)}: t = 0, 1, \ldots\}$.
\end{algorithm}

In this article, we will focus on the online part, namely ALGO2. The complexity of ALGO1 will be increasingly less important as $t$ increases. ALGO1 can be simply chosen as any state-of-the-art algorithm such as the interior-point method. Note that in Online Newton Step \citep{hazan2006logarithmic}, $A^{(0)} = \eps I$ is a scaled identity matrix and hence $x^{(0)} = \frac{1}{n}\textbf{1}$. 

\subsection{KKT Condition Within A Step}
We first consider the problem for a given $t$. The aim is to minimize $\frac{1}{2}x^{T}A x - \c^{T}x$ over $\Delta_{n}$, where $A$ and $\c$ are abbreviation of $A^{(t)}$ and $\c^{(t)}$. By strong duality, it is equivalent to minimize the Lagrangian form
\begin{equation}\label{eq:lag}
L(x; \mu_{0}, \mu) = \frac{1}{2}x^{T}A x - \c^{T}x + \mu_{0}(1 - \one^{T}x) -\mu^{T}x
\end{equation}
where $\mu_{0}$ and $\mu$ are Lagrangian multipliers with constraint $\mu_{i}\ge 0$ for $i = 1, \ldots, n$. Denote $S_{x}$ by the support of vector $x$. To be concise, the subscript $x$ is suppressed in the following context. KKT condition together with Slater's condition implies that $(x, \mu_{0}, \mu)$ is the solution of (\ref{eq:lag}) if and only if
\begin{align}
& \A x - \mu_{0}\one - \mu - \c= 0; \label{eq:solution}\\
& \one^{T}x = 1;\label{eq:sumone}\\ 
& \mu_{i}x_{i} = 0, \mu_{i}\ge 0, x_{i}\ge 0, \forall i = 1,\ldots, n\label{eq:comp_slack}.
\end{align}
Here (\ref{eq:comp_slack}) is dubbed \emph{complementary slackness} condition. The definition of $S = \supp(x)$ entails that $\xc = 0$, and (\ref{eq:comp_slack}) further implies that $\mus = 0$. Then the condition (\ref{eq:solution}) can be reformulated as
\begin{align*}
&\mat{\Ass}{\Asc}{\Acs}{\Acc}\com{\xs}{0}\\ 
= &\mu_{0}\com{\ones}{\onec} + \com{0}{\muc} + \com{\cs}{\cc}
\end{align*}
By separating $\S$ and $\S^{c}$, we have the following equations for $\xs$ and $\muc$.
\begin{align}
  \xs &= \mu_{0}\Ass^{-1}\ones + \Ass^{-1}\cs; \label{eq:xs}\\
  \muc &= \Acs\xs  - \mu_{0}\onec - \cc \nonumber\\
& = -\mu_{0}(\onec - \Acs\Ass^{-1}\ones) - (\cc - \Acs\Ass^{-1}\cs). \label{eq:muc}
\end{align}
The other parameter $\mu_{0}$ can be solved from (\ref{eq:sumone}) and (\ref{eq:xs}). In fact, 
\[1 = \one^{T}x = \ones^{T}\xs = \mu_{0}\ones^{T}\Ass^{-1}\ones + \ones^{T}\Ass^{-1}\cs\]
which implies that
\begin{equation}\label{eq:muzero}
\mu_{0} = \frac{1 - \ones^{T}\Ass^{-1}\cs}{\ones^{T}\Ass^{-1}\ones}.
\end{equation}
In summary, the quadruple $(\S, \xs, \muc, \mu_{0})$ which solves (\ref{eq:xs})-(\ref{eq:muzero}) produces the unique solution of (\ref{eq:lag}). Moreover, given the correct support $\S$, we can uniquely solve the other three parameters. Thus, determining $\S$ is the key part in this problem.

\subsection{\OurMethod Algorithm}

\subsubsection{Construction of Homotopy Continuation}\label{subsubsec:construct}
Based on the above argument, the problem is reduced to updating
support $\S$ with $\A$ replaced with $\A + \g\g^{T}$, and $\c$ replaced by $\c + \l$, where $\g, \l$ are shorthand notations of $g^{(t)}$ and $\c^{(t)} - \c^{(t - 1)}$. Heuristically, $\S$ will not be significantly disturbed when $\g, \l$ are small perturbations. However, in real problems, there is usually no such constraints on $\g$. Instead, we can consider a homotopy from $(\A, \c)$ to $(\A + \g\g^{T}, \c + \l)$. The most natural one is $(\A +\lam \g\g^{T}, \c + \ulam\l)$ with $\lam, \ulam\in [0, 1]$. In other words, if we denote $x(\lam, \ulam)$ be the solution of (\ref{eq:lag}) with $(\A, \c)$ replaced by $(\A + \lambda \g\g^{T}, \c + \l)$, then $x(0, 0)$ is the solution in the last step and $x(1, 1)$ is the solution after the update. The idea of homotopy continuation method is to calculate $x(\lam, \ulam)$ over a path linking $(0, 0)$ to $(1, 1)$. Theoretically, any path suffices and the goal is to find a path which leads to a simple computation. In this article we will consider the Manhattan path from $(0, 0)$ to $(1, 1)$, namely the union of three segments: $\{(z, 0): z \in [0, 1]\}$ and $\{(1, z): z\in [0, 1]\}$. In other words we first minimize 
\[H^{(1)}(\lam)\triangleq \frac{1}{2}x^{T}(\A + \lam \g\g^{T})x - \c^{T}x\] 
for each $\lam\in [0, 1]$ and then minimize 
\[H^{(2)}(\ulam) \triangleq\frac{1}{2}x^{T}(\A + \g\g^{T})x - (\c + \ulam\l)^{T}x\] 
for each $\ulam\in [0, 1]$. 

Although the problem is augmented, the update is efficient since the support $\S$ is shown to be a piecewise constant set on the path and explicit formulas, namely (\ref{eq:xs}) - (\ref{eq:muzero}), can be used to compute $(\xs, \muc, \mu_{0})$ directly when $S$ is fixed. In fact, the triple $(\xs, \muc, \mu_{0})$ is a simple function of $(\lam, \ulam)$ as shown in the following theorem.
\begin{theorem}\label{thm:update}
~
  \begin{enumerate}
  \item   For a given $\ulam$, there exists vectors $u_{1}, u_{2}\in \R^{n + 1}$ and scalars $D_{1}, D_{2}\in \R$, which only depend on $\S$, such that
\begin{equation}\label{eq:lambda_update}
\lb\begin{array}{c}\xs(\lam) \\ -\muc(\lam) \\ \mu_{0}(\lam)\end{array}\rb = \frac{u_{1} - u_{2}\lam}{D_{1} - D_{2}\lam}.
\end{equation}
  \item For given $\lam$, there exists vectors $\utd{u}_{1}, \utd{u}_{2}\in \R^{n + 1}$, which only depend on $\S$, such that
\begin{equation}\label{eq:tlambda_update}
\lb\begin{array}{c}\xs(\ulam) \\ -\muc(\ulam) \\ \mu_{0}(\ulam)\end{array}\rb = \utd{u}_{1} - \utd{u}_{2}\ulam.
\end{equation}
  \end{enumerate}
\end{theorem}

\begin{proof} 
\begin{enumerate}[leftmargin=4mm]
  \item The proof is quite involved and we relegate it into Theorem \ref{thm:var_lambda} in Appendix \ref{app:varyA}. The theorem also gives the exact formula of $u_{1}, u_{2}, D_{1}, D_{2}$.
  \item By \eqref{eq:muzero}, we have
\[\mu_{0}(\ulam) = \frac{1 - \ones^{T}\Ass^{-1}(\cs + \ulam \ls)}{\ones^{T}\Ass^{-1}\ones} = \mu_{0}(0) - \frac{\ones^{T}\Ass^{-1}\ls}{\ones^{T}\Ass^{-1}\ones}\ulam.\]
Then it follows from \eqref{eq:xs} that
\begin{align*}
&\xs(\ulam) = \mu_{0}(\ulam)\Ass^{-1}\ones + \Ass^{-1}(\cs + \ulam\ls)\\ 
= & \xs(0) - \lb \frac{\ones^{T}\Ass^{-1}\ls}{\ones^{T}\Ass^{-1}\ones}\Ass^{-1}\ones - \Ass^{-1}\ls\rb\ulam.
\end{align*}
Similarly, by \eqref{eq:muc}, we obtain that 
\[-\muc(\ulam) = -\muc(0) - \bigg( \frac{\ones^{T}\Ass^{-1}\ls}{\ones^{T}\Ass^{-1}\ones} (\onec - \Acs\Ass^{-1}\ones)\] 
\[- (\lc - \Acs\Ass^{-1}\ls)\bigg)\ulam.\]
  \end{enumerate}
\end{proof}

\subsubsection{Update of Support}
Once $\S = \S(\ulam)$ is obtained for all $\ulam$, the solution path can be efficiently solved by Theorem \ref{thm:update}. Heuristically, $\S$ is piecewise constant and the task is reduced to find the next $\ulam$ that $\S(\ulam)$ changes. We consider the update of $\S$ in optimizing $H^{(1)}(\lam)$. The update of $\S$ in optimizing $H^{(2)}(\ulam)$ can be obtained in the same way. 

For a given $\lam_{0}\in [0, 1]$, if
$\xs(\lam_{0}) > 0$ and $\muc(\lam_{0}) > 0$, then
(\ref{eq:lambda_update}) implies that there exists $\eta > 0$, such
that for any $\lam \in (\lam_{0} - \eta, \lam_{0} + \eta)$,
both $\xs(\lam)$ and $\muc(\lam)$ remains positive by setting
$\S(\lam) = \S(\lam_{0})$. Since (\ref{eq:xs})-(\ref{eq:muzero})
are sufficient and necessary, we conclude that $\S(\lam) =
\S(\lam_{0})$. This argument remains valid until an entry of either
$\xs$ or $\muc$ hits zero. Denote $j$ by the index of this entry. In
the former case, $j$ leaves $\S$ and $\S$ is updated to
$\S\setminus\{j\}$. In the latter case, $j$ enters into $\S$ and $\S$ is
updated to $\S\cup \{j\}$. The other three parameters are then updated
correspondingly by Theorem \ref{thm:update}. Theorem
\ref{thm:update_support} formalizes the above claim. The proof is omitted since it is a direct consequence of sufficiency and necessity of KKT conditions (\ref{eq:xs})-(\ref{eq:muzero}).

\begin{theorem}\label{thm:update_support}
For any given $\lam_{0} \ge 0$, let $\lam^{\mathrm{new}}$ be the
next smallest $\lam$ such that one entry of either $\xs$ or $\muc$ hits 0, i.e.
\[\lam^{\mathrm{new}} = \mathrm{min}_{+}\left\{\frac{u_{1i}}{u_{2i}}: i = 1, 2, \ldots, n\right\},\]
where $u_{1}$ and $u_{2}$ are defined in (\ref{eq:lambda_update}) and $\mathrm{min}_{+}$ evaluates the minimum positive number in the set and defined to be $\infty$ if all elements are non-positive. Then $\S(\lam) \equiv S(\lam_{0})$ for $\lam\in [\lam_{0}, \lam^{\mathrm{new}})$. Further, let 
\[I_{1} = \{i\in \S: u_{1i} = u_{2i}\lam^{\mathrm{new}}\},\]
\[I_{2} = \{i\in \S^{c}: u_{1i} = u_{2i}\lam^{\mathrm{new}}\},\]
then $\S(\lam^{\mathrm{new}})$ is updated by
\[S(\lam^{\mathrm{new}}) = (S(\lam_{0})\setminus I_{1}) \cup I_{2}.\]
\end{theorem}
\begin{remark}\label{rem:update_one}
  According to our experience, $I_{1}\cup I_{2}$ at most contains one element. In other words, $\S$ is updated by one element at each iteration.
\end{remark}

In summary, the algorithm starts from $\lam = 0$ and searches for the next smallest $\lam$ such that one entry of $\xs$ or $\muc$ hits zero, then updates $\lam$ as well as the quadruple $(S, \xs, \muc, \mu_{0})$. The procedure is repeated until $\lam$ crosses 1. In other words, there exist a sequence $0 = \lam_{0} < \lam_{1} < \ldots < \lam_{k} = 1$, which we call \emph{turning points}, such that $\x(\lam)$ has the same support between any two consecutive turning points and the value can be calculated by Theorem \ref{thm:update}. A counterpart of Theorem \ref{thm:update_support} can be established for $\ulam$. The whole task reduces to finding all turning points and we call this procedure \OurMethod algorithm. The complexity of \OurMethod algorithm is determined by both the number of turning points and the complexity of the update between two consecutive turning points. For compact notation, we define $v$ as a $n\times 1$ vector with 
\[v_{\S} = \xs, \quad v_{\S^{c}} = -\muc.\]
\noindent To be more clear, we state the main steps in Algorithm \ref{algo:homotopy} for optimizing $H^{(1)}(\lam)$ holding $\ulam=0$. As a convention, the minimum of an empty set is set to be infinity (line 3). The algorithm for optimizing $H^{(2)}(\ulam)$ holding $\lam=1$ can be written in the same way as Algorithm \ref{algo:homotopy} by changing $\lam$ into $\ulam$. 
\begin{algorithm}
  \caption{Main steps of \OurMethod algorithm in optimizing $H^{(1)}(\lam)$}\label{algo:homotopy}
\textbf{Inputs: } parameters $A, y, \c, g$; initial optimum $x$ (corresponding to $A$)

\textbf{Procedure: }
\begin{algorithmic}[1]
  \STATE Initialize $\lam\gets 0, \S\gets \supp(x)$;
  \WHILE{$\lam < 1$}
  \STATE $\lam = \min\{\lam_{1} > \lam: v_{i}(\lam_{1}) = 0\mbox{ for some }i\}$;
  \STATE $I_{1}\gets \{i\in \S: v_{i}(\lam) = 0\}$;
  \STATE $I_{2}\gets \{i\in \S^{c}: v_{i}(\lam) = 0\}$;
  \IF {$\lam \le 1$}
  \STATE $\S\gets (\S\setminus I_{1})\cup I_{2}$;
  \ELSE
  \STATE $\lam\gets 1$;
  \ENDIF
  \STATE $(\xs, \muc, \mu_{0})\gets (\xs(\lam), \muc(\lam), \mu_{0}(\lam))$ via (\ref{eq:xs})-(\ref{eq:muzero}).
  \ENDWHILE
\end{algorithmic}
\textbf{Output: } $(S, \xs, \muc, \mu_{0})$.
\end{algorithm}

\subsection{Implementation and Complexity Analysis}\label{subsec:complexity}
Algorithm \ref{algo:homotopy} presents the main idea without the implementation details. Although we can implement Algorithm \ref{algo:homotopy} by directly computing quantities, e.g. $u_{1}, u_{2}$, in every step to find the next turning point as in line 3 and also directly computing the iterates via (\ref{eq:xs})-(\ref{eq:muzero}) as in line 11, it is fairly inefficient since many quantities appear in several computation steps and we can store them to save the computation. A careful derivation in Appendices \ref{app:varyA} and \ref{app:varyc} shows that the computation complexity is indeed low. For example, although $u_{1}$ and $u_{2}$ involves $\Ass^{-1}$, there is \emph{no need} to calculate the matrix inverse directly. Theorem \ref{thm:complexity} summarizes the complexity for optimizing $H^{(1)}(\lam), H^{(2)}(\ulam)$ separately. As a convention, we assume the scalar-scalar multiplication takes a unit time and ignore the addition for simplicity when computing the complexity. Since the real implementation is involved, we state it as well as the proof of theorem \ref{thm:complexity} in Appendices \ref{app:varyA} and \ref{app:varyc} for two cases separately. 
\begin{theorem}\label{thm:complexity}
In step $t$, denote by $k_{\A}, k_{\c}$ the number of turning points in optimizing $H^{(1)}(\lam)$ and $H^{(2)}(\ulam)$. Further let $s$ be the maximum support size over the path of $(\lam, \ulam)$ and $s_{*}$ by the size of union of all supports from step $1$ to step $t$. Let $C_{jt}$ be the computation cost of \OurMethod algorithm in optimizing $H^{(j)}$, then
\begin{enumerate}
\item $C_{1t} = ns_{*} + ns(3k_{\A} + 1) + n(12k_{\A} + 2) + O(k_{\A})$;
\item $C_{2t} = ns(2k_{\c} + 1) + n(6k_{\c} + 1) + O(k_{\c})$.
\end{enumerate}
\end{theorem}
It is clear that the algorithm adapts to the sparsity when optimizing both $H^{(1)}(\lam)$ and $H^{(2)}(\ulam)$. For each step, the complexity is $\mathcal{O}\left (ns(k_A + k_r)\right )$, upper bounded by $\mathcal{O}(n^2(k_A + k_r))$ for the dense case, which is the same as other algorithms due to the inevitable multiplication of $A$ by $x$. In some special regimes, the solution is guaranteed to be sparse with high probability, e.g., the data matrix is randomly generated form a certain distribution such as uniform and exponential distributions \cite{chen13,chen2015new}. We also observe this in our experiments; See Section \ref{sec:experiments} for details.

\subsection{Number of Turning Points}\label{subsec:turningpoints}
Let $\S_{t}$ be the support of the optimum and $k_{t}$ be the number of total turning points, then we can derive a generic bound that
\begin{equation}\label{eq:lower_bound_k}
k_{t}\ge |\S_{t}\setminus \S_{t - 1}| + |\S_{t - 1}\setminus \S_{t}|
\end{equation}
provided that only one element is added to or removed from the support at each update; see Remark \ref{rem:update_one}. This is because it requires at least $|\S_{t - 1}\setminus \S_{t}|$ steps to pop out the elements in $\S_{t-1}\setminus \S_{t}$ and $|\S_{t}\setminus \S_{t - 1}|$ steps to push in the elements in $\S_{t}\setminus \S_{t-1}$ to translate $\S_{t - 1}$ into $\S_{t}$. 

On the other hand, suppose that no other coordinates than those in $\S_{t}\cup \S_{t - 1}$ enter into the support in the path, then 
\begin{equation}\label{eq:upper_bound_k}
k_{t} = |\S_{t}\setminus \S_{t - 1}| + |\S_{t - 1}\setminus \S_{t}|.
\end{equation}
Heuristically, the equation \eqref{eq:upper_bound_k} should hold since if a coordinate, not in $\S_{t}\cup \S_{t-1}$, entered into the support on the path, it must be popped out before the end, which, however, should be rare to happen. For both synthetic data and real data in section \ref{sec:experiments}, we observe that there are at least $95\%$ of steps with $k_{t}$ satisfying \eqref{eq:upper_bound_k} and over $99\%$ of steps with $k_{t}\le |\S_{t}\setminus \S_{t - 1}| + |\S_{t - 1}\setminus \S_{t}| + 6$, i.e. with at most $3$ outside coordinates entered into the path. Thus, \eqref{eq:upper_bound_k} is a highly reliable result for $k_{t}$.

As a direct consequence of \eqref{eq:upper_bound_k}, \OurMethod algorithm is efficient when the support changes slowly in which case $k_{t}$ is small. In addition, if the solution is sparse, then a rough bound suggests that $k_{t}\le |\S_{t}| + |\S_{t - 1}|$ is small. These phenomena are observed in various situations (see section \ref{sec:experiments}) and \eqref{eq:upper_bound_k} explains the good performance of \OurMethod algorithm. 

In general, the worst-case bound for the number of turning points can be exponential as \citet{mairal2012complexity,gartner2009exponential}  pointed out for Lasso and SVM respectively. But the number of turning points is usually not large in practice. The same issue appears in Simplex method for linear programming. Although it is known that the worst case complexity is $2^n$, it usually converges in $\mathcal{O}(n)$ operations;See \cite{bertsimas1997introduction}.

\section{Experiments}\label{sec:experiments}
In this section, we compare the performance of HONES with SPG and the interior-point method on both real and synthetic data. We implement \OurMethod in MATLAB \footnote{Code available at https://github.com/Elric2718/HOP.} and implement SPG \footnote{https://www.cs.ubc.ca/~schmidtm/Software/minConf.html} and interior-point method\footnote{\textsc{quadprog} function in MATLAB} by using existing code.  To make a fair comparison, SPG uses the solution to step $t$ as the warm start for the solution to step $t + 1$. To evaluate the performance, we display the cumulative running time as a measure of efficiency. All experiments are conducted on a machine with 3 GHz Intel Core i7 processor, OS X Yosemite system and Matlab 2015a. \footnote{Here we exclude the running time incurred by updating $A^{(t-1)}$ to $A^{(t)}$ since this is unrelated to the optimization and is unavoidable for whatever algorithms.}

\subsection{Universal Portfolio Management}\label{subsec:ons}
\citep{hazan2006efficient} proposed a version of Online Newton Step (Figure 3.7) that is equivalent to \eqref{eq:stqp} with 
\[A^{(0)} = I, \quad g^{(t)} = \frac{\gamma_{t}}{x_{t}^{T}\gamma_{t}}, \quad \c^{t} = \frac{1}{4}\sum_{\tau = 1}^{t}\frac{\gamma_{t}}{x_{t}^{T}\gamma_{t}}.\]
We apply our algorithm on two datasets from NYSE and NASDAQ\footnote{Data available at https://github.com/Elric2718/HOP.}, with daily stock price data from Jan. 3, 2005 to May. 13, 2016. The NYSE dataset contains $1544$ stocks and NASDAQ dataset contains $1101$ stocks. This differs from classical studies where at most hundreds of stocks, such as S\&P500, are incorporated. Still, we should emphasize that for some financial institutions like hedge funds, the number of base assets is huge and the computation efficiency becomes important when the trading frequency is high. Here we consider a large number of stocks to show the potential of \OurMethod algorithm in optimizing a large basket of assets. 

The cumulative running time, measured in seconds, is reported in Figure \ref{fig:ons_cumtime}. The interior-point method is quite inefficient as the running time for 10 steps (0.04 epochs) exceeds the total running time of \OurMethod and SPG. Thus the proposal by \cite{hazan2006logarithmic} is not desirable. In addition, \OurMethod is much more efficient than both SPG, especially in the more volatile case (NASDAQ). In fact, \OurMethod achieves a $2\times$ speedup on NYSE data and a $6\times $ speedup on NASDAQ data! By excluding the first 4 epochs, \OurMethod even achieves a $12.5\times$ speedup on NASDAQ data.

\begin{figure}
  \begin{center}
    \includegraphics[width = 0.8\textwidth]{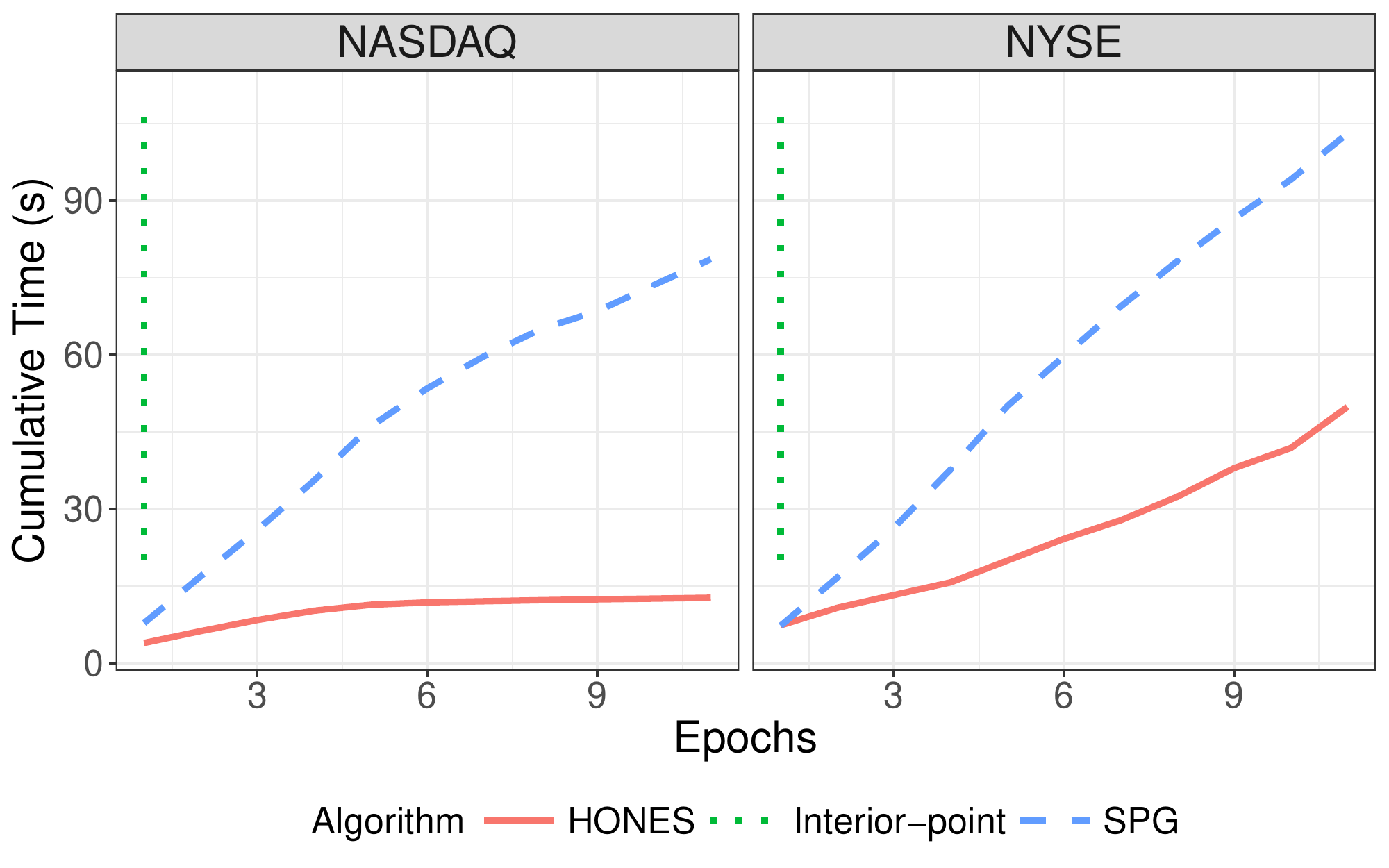}
  \end{center}
  \caption{Cumulative running time of \OurMethod, SPG and interior-point method on NYSE and NASDAQ dataset for universal portfolio management. Each epoch has 252 measurements.}
  \label{fig:ons_cumtime}
\end{figure}

To explain the different behavior on two datasets, we report the average solution sparsity in Table \ref{tab:ons}. Recall that \OurMethod is accurate in every step and we confirm this by checking the KKT condition for each solution. Surprisingly, the solutions are generally sparse for both datasets, as suggested by existing theory \citep{chen13, chen2015new}. It is not surprising that the solutions are sparser on NASDAQ due to the high volatility. As indicated by our theory, the efficiency gain should be more significant in this case.

\begin{table}[h]
  \centering
  \caption{Characteristics of \OurMethod for universal portfolio management on NYSE and NASDAQ data. (Top) solution Sparsity on NYSE and NASDAQ datasets, including the average, standard error, maximum and minimum of the support size; (Bottom) distribution of $e_{t}$, the number of excess turning points, on NYSE and NASDAQ datasets.}\label{tab:ons}
  \begin{tabular}{ccccc}
    \toprule
    Dataset & \multicolumn{4}{c}{sparsity} \\
    \midrule
     & mean & std. & max & min\\
    NYSE & 16.0 & 8.6& 98 & 4\\
    NASDAQ & 6.46 & 3.7 & 64 & 2 \\ 
    \toprule
    Dataset & \multicolumn{2}{c}{proportion of zeros} & \multicolumn{2}{c}{quantiles}\\
     & \multicolumn{2}{c}{} & 99\%& 99.9\%\\
    NYSE & \multicolumn{2}{c}{65.8\%} & 8 & 22\\
    NASDAQ & \multicolumn{2}{c}{85.4\%} & 4 & 11 \\
    \bottomrule
  \end{tabular}
\end{table}

Finally, we examine our conjecture in Section \ref{subsec:turningpoints} on the number of turning points. As explained there, a benchmark for $k_{t}$ is $|\S_{t}\setminus \S_{t - 1}| + |\S_{t - 1}\setminus \S_{t}|$. We refer to $e_{t} = (k_{t} - |\S_{t}\setminus \S_{t - 1}| - |\S_{t - 1}\setminus \S_{t}|) / 2$ as the number of \emph{excess turning points}; see Section \ref{subsec:turningpoints} for details. For each synthetic dataset, we report the proportion of zero $e_{t}$ in Table \ref{tab:ons}. It is clear that most steps (65.8\% for NYSE data and 85.4\% for NASDAQ data) are predicted by our conjecture and almost all steps (over 99\% for both data) are not far away from our conjecture. This indicates that $k_{t}$ is an accurate proxy for the number of turning points.

\subsection{Synthetic Data}

\begin{figure}[H]
  \begin{center}
    \includegraphics[width = 0.8\textwidth, height = 0.6\textwidth]{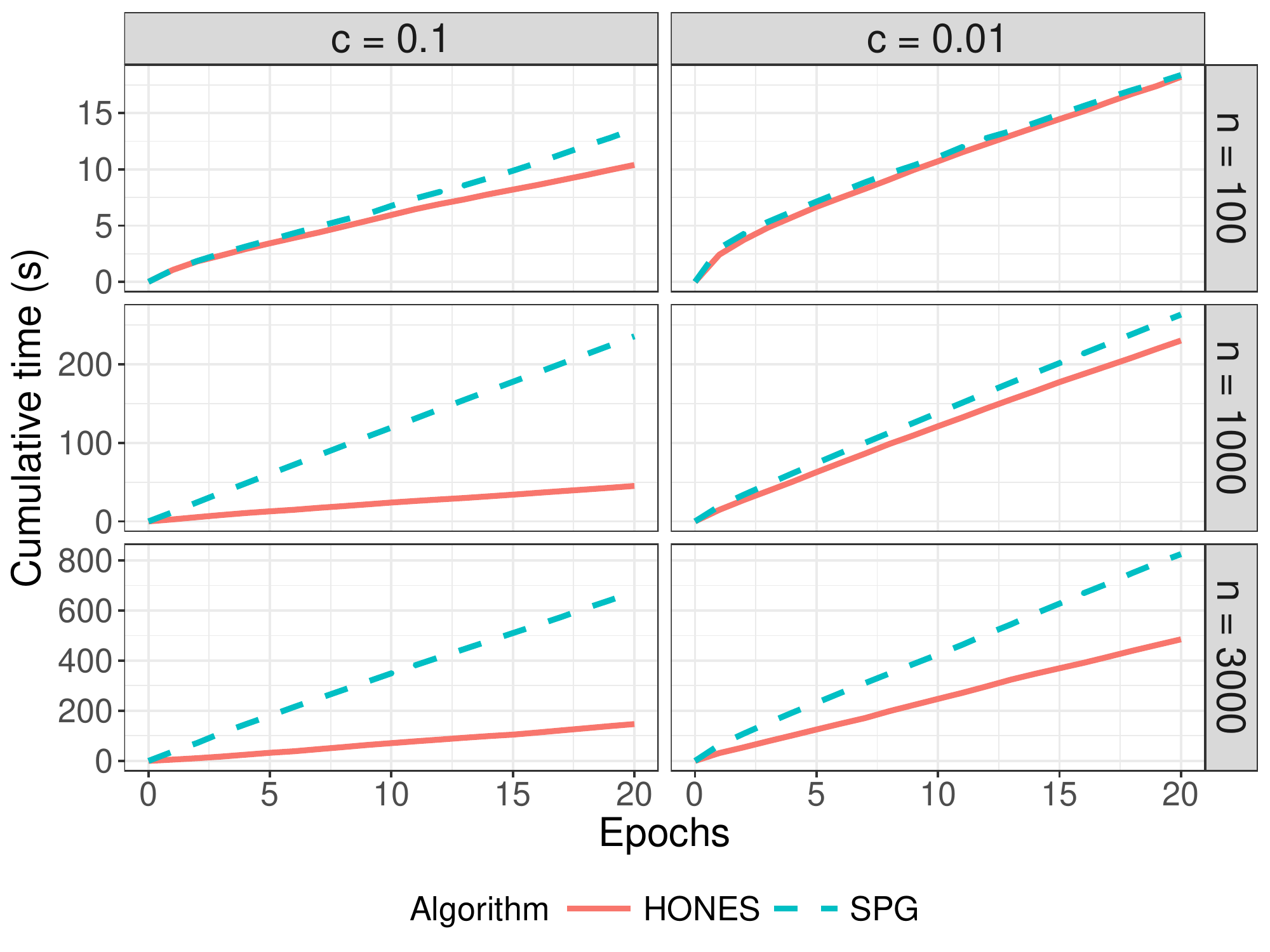}
  \end{center}
  \caption{Cumulative running time of \OurMethod and SPG on synthetic datasets. Each row corresponds to a dimension $n$ and each column corresponds to a factor $c$.}\label{fig:sim}
\end{figure}

As discussed in Section \ref{sec:intro}, the problem \eqref{eq:stqp} is indeed more general than universal portfolio management. To examine our algorithm comprehensively, we consider \eqref{eq:stqp} under other setups. First we consider the problem with the following structure on synthetic data:
\begin{equation}\label{eq:stqp2}
\min_{x\in \Delta_{n}}\frac{1}{2}(x - y)^{T}A^{(t)}(x - y).
\end{equation}
Without the superscript $t$, this problem is called \emph{standard quadratic programming problem} and has attracted the attention in various fields, e.g. \cite{bomze98, scozzari08, bomze08}. It is of particular interest to study the case where $A^{(t)}$ is a random matrix generated from some distribution. For instance, \citet{chen13} consider a Wigner matrix $A$ with $\{A_{ij}: 1\le i\le j\le n\}$ being i.i.d. random variables and $A_{ji} = A_{ij}$. In this article, we consider another important class of matrices in random matrix theory --- covariance matrix of rectangular matrices with i.i.d. entries, i.e. $A^{(t)} = \sum_{s=1}^{t}g^{(s)}(g^{(s)})^{T}$, where $g^{(s)}\in \R^{n}$ has i.i.d. Gaussian entries; see \citet{bai10} for more discussion. In this case the matrix flow $\{A^{(t)}: t = 1, 2,\ldots \}$ satisfies \eqref{eq:update}. To avoid singularity, we set $A^{(0)} = \eps I$ where $\eps = 10^{-4}$ is a small positive number. Then it is easy to see that $A^{(t)}$ is non-singular for all $t$ with probability $1$ so that the solution $x^{(t)}$ is unique. The vector $y$ governs the sparsity of the solution. To see this, consider the isotropic case where $A = I$, the solution of \eqref{eq:stqp2} is the projection of $y$ onto the simplex. If $y$ is a zero vector, the optimum is a dense vector with all entries $\frac{1}{n}$. In contrast, if $y$ has large entries, the simplex constraint will pull the optimum towards that direction and forces the other entries to be zero , in which case the solution is sparse. The same phenomenon is observed in anisotropic case as will be shown below.

\begin{table}
  \centering
  \caption{Solution Sparsity and overall computation gain of \OurMethod over SPG on synthetic datasets. The first two columns correspond to the dimension and the factor $c$; the third column gives the mean of support size with its standard deviation (in the parentheses); the fourth column gives the maximum support size along the path; the last two columns show the ratio of overall running time between SPG and \OurMethod.}\label{tab:sim1}
  \begin{tabular}{ccccc}
    \toprule
    \multicolumn{2}{c}{scenarios} & \multicolumn{2}{c}{sparsity} & speed ratio\\
    \midrule
    $n$ & $c$ & mean (s.e.) & max& \\
    \midrule
    100 & 0.01 & 81.3 (3.5) & 88 & 0.97 \\
    1000 & 0.01 & 158.3 (30.6) & 190 & 1.13 \\
    3000 & 0.01 & 146.4 (34.9) & 225 & 1.68 \\
    100 & 0.1 & 19.5 (2.3) & 26 & 1.34 \\
    1000 & 0.1 & 22.0 (5.3) & 32 & 5.26 \\
    3000 & 0.1 & 18.3 (4.2) & 27 & 4.51 \\
    \bottomrule
  \end{tabular}
\end{table}

Our goal is to explore the scalability, in terms of the dimension, and the adaptivity to solution sparsity of the algorithms. For the aspect of the dimension, we consider three dimensions: $\{100, 1000, 3000\}$; for the aspect of the sparsity,  we set $y = cy_{0}$ with $y_{0}$ generated from $N(0, I_{n\times n})$ and $c\in \{0.01, 0.1\}$. For each case, we set the total number of steps as $5000$ and treat every 250 steps as an epoch (20 epochs in total). Similar to the previous case, we report the cumulative running time in Figure \ref{fig:sim} and other information in Table \ref{tab:sim1} and Table \ref{tab:sim2}. Here we exclude the interior-point method since it is too slow.

\begin{table}[hbp]
  \centering
  \caption{Distribution of $e_{t}$, the number of excess turning points, on synthetic datasets. The first two columns correspond to the dimension and the factor $c$; the third column gives the proportion of zero $e_{t}$; the last two columns give the 99\% and 99.9\% quantiles of $e_{t}$.}\label{tab:sim2}
  \begin{tabular}{ccccc}
    \toprule
    \multicolumn{2}{c}{scenarios} & proportion of zeros & \multicolumn{2}{c}{quantiles}\\
    \midrule
    $n$ & $c$ &  & 99\%& 99.9\%\\
    \midrule
    100 & 0.01 & 99.8\% & 0 & 1\\
    1000 & 0.01 & 98.9\% & 1 & 4\\
    3000 & 0.01 & 98.7\% & 1 & 6\\
    100 & 0.1 & 99.9\% & 0 & 0 \\
    1000 & 0.1 & 99.8\% & 0 & 1\\
    3000 & 0.1 & 99.8\% & 0 & 1 \\
    \bottomrule
  \end{tabular}
\end{table}

First we notice that \OurMethod is more scalable in high dimension. Moreover, as expected, a larger $c$ gives sparser solutions along the path and \OurMethod significantly outperforms SPG when the solution is sparse ($c = 0.1$) especially for large-scale problems ($n = 1000, 3000$), in which \OurMethod is over 4.5 times faster than SPG. When the solution is not sparse ($c = 0.01$), \OurMethod is similar to SPG in small-scale problem ($n = 100$) and increasingly more efficient when the size of the problem grows. Finally the results in Table \ref{tab:sim2} show that our conjecture in Section \ref{subsec:turningpoints} is extremely accurate in this setting.

\subsection{Markowitz Portfolio Selection}
In this Subsection, we consider the application of HONES algorithm on sequential Markowitz portfolio selection problem. In this problem, the vector of stock prices is assumed to be a random vector with mean $\mu$ and covariance matrix $\Sigma$ and the investor observes a realization from this distribution. The goal is to minimize the risk, measured by the variance $x^{T}\Sigma x$ of a given portfolio while maintaining a reasonably high average return $x^{T}\mu$. The problem is usually formulated as follows:
\[\min_{x_{t}\in \R^{n}} \frac{1}{2}x_{t}^{T}\Sigma x_{t} - \lambda x_{t}^{T}\mu.\]
For simplicity we assume $\lambda = 0$. In practice, $\mu$ and $\Sigma$ are unknown and one has to replace them by estimators. The most natural estimators $\hat{\Sigma}^{(t)}$ and $\hat{\mu}^{(t)}$ are the sample covariance matrix and the sample mean, i.e.
\[\hat{\mu}^{(t)} = \frac{1}{t}\sum_{s=1}^{t}w^{(s)}, \hat{\Sigma}^{(t)} = \frac{1}{t}\sum_{s = 1}^{t}(w^{(s)} - \hat{\mu}^{(s)})(w^{(s)} - \hat{\mu}^{(s)})^{T},\]
where $w^{(t)}$ is the vector of daily gains, measured by the entrywise log return $\log (\gamma_t)$, of all assets of interest. Via some algebra, it can be shown that 
the problem is equivalent to \eqref{eq:stqp} with 
\[A^{(t)} = t\hat{\Sigma}^{(t)}, \c^{(t)} = t\hat{\mu}^{(t)}, g^{(t)} = \sqrt{\frac{t - 1}{t}}\lb w^{(t)} - \hat{\mu}^{(t - 1)}\rb.\]
We should emphasize that the solution in this way is optimal from hindsight, which is different from the notions in online learning regret minimization. Nonetheless, it is an interesting and important problem in the context of back testing and risk management since the result can reveal the hidden structure of the assets; see \cite{brodie2009sparse,fan2008asset,fan2012vast} for more details. 

\begin{figure}[H]
  \begin{center}
    \includegraphics[width = 0.8\textwidth]{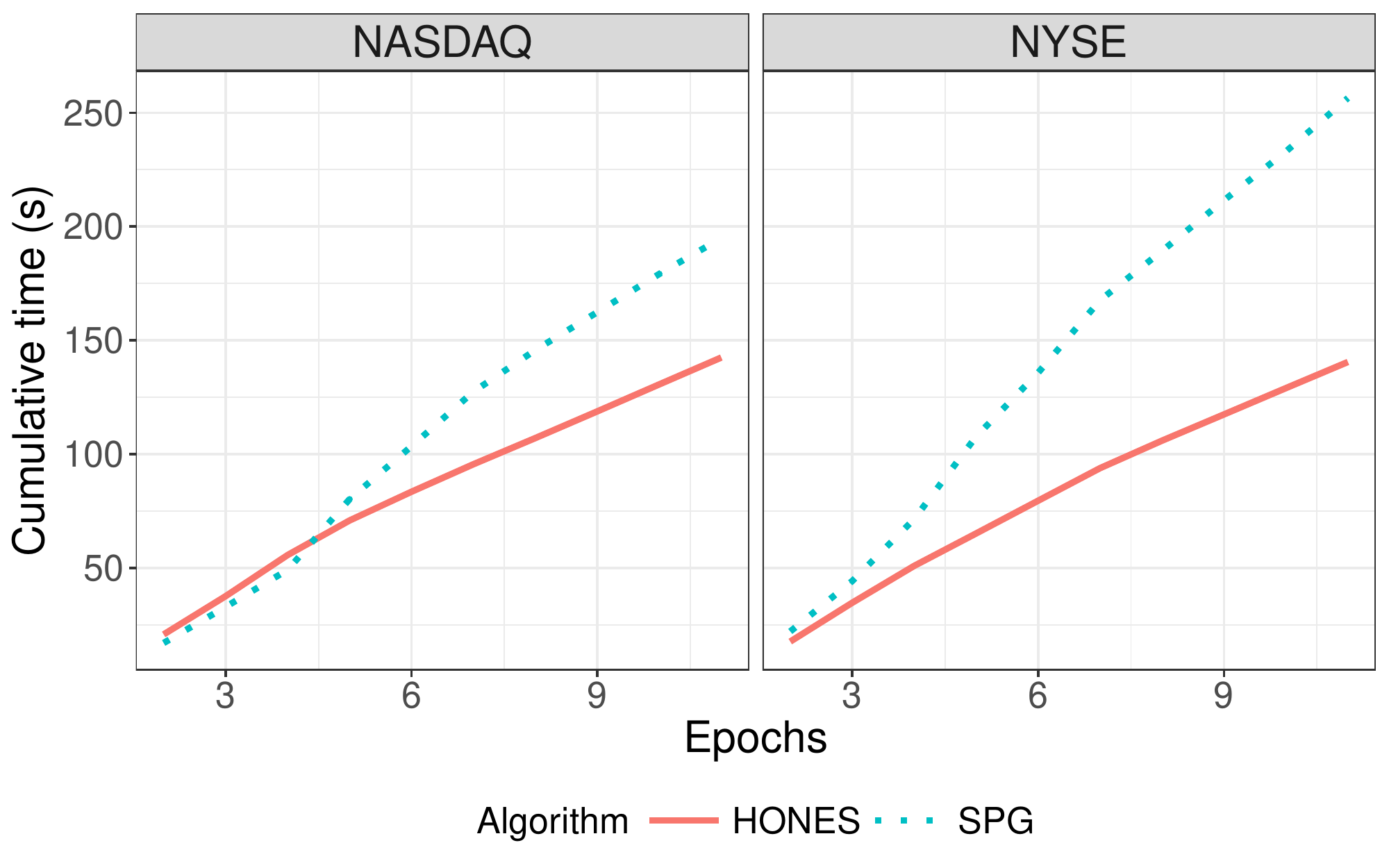}
  \end{center}
  \caption{Cumulative running time of \OurMethod and SPG on NYSE and NASDAQ dataset for Markowitz portfolio management. Each epoch has 252 measurements.}
  \label{fig:real}
\end{figure}

Similar to the Section \ref{subsec:ons}, we report the cumulative running time of \OurMethod and SPG in Figure \ref{fig:real} and report other information in Table \ref{tab:real}. Again, \OurMethod is more efficient than SPG. We also try the interior-point method on each dataset. It is 67 times slower than \OurMethod on NYSE dataset and 31 times slower than \OurMethod on NASDAQ dataset. Finally, the conjecture on the number of turning points is also validated by the results in the bottom panel of Table \ref{tab:real}.

\begin{table}[h]
  \centering
  \caption{(Top) solution Sparsity on NYSE and NASDAQ datasets for Markowitz portfolio management, including the average, standard error, maximum and minimum of the support size; (Bottom) distribution of $e_{t}$, the number of excess turning points, on NYSE and NASDAQ datasets.}\label{tab:real}
  \begin{tabular}{ccccc}
    \toprule
    Dataset & \multicolumn{4}{c}{sparsity} \\
    \midrule
     & mean & s.e. & max & min\\
    NYSE & 49.8 & 22.7& 108 & 30\\
    NASDAQ & 148.9 & 17.8 & 192 & 127 \\ 
    \toprule
    Dataset & \multicolumn{2}{c}{proportion of zeros} & \multicolumn{2}{c}{quantiles}\\
     & \multicolumn{2}{c}{} & 99\%& 99.9\%\\
    NYSE & \multicolumn{2}{c}{97.9\%} & 1 & 10\\
    NASDAQ & \multicolumn{2}{c}{96.2\%} & 3 & 22 \\
    \bottomrule
  \end{tabular}
\end{table}

\section{Conclusion}\label{sec:conclusion}
In this article, we propose an efficient algorithm \OurMethod to solve the sequential generalized projection problem \eqref{eq:stqp} with rank-one update \eqref{eq:update}, appeared as the building block and the bottleneck of Online Newton Step. \OurMethod is a homotopy continuation method that interpolates the consecutive objectives. By a careful derivation, we calculate the exact number of atomic operations, up to an additive constant (Theorem \ref{thm:complexity}) and show that \OurMethod has a good performance when the support of the solution changes slowly with time or is sparse as in many applications. We also provide a heuristic conjecture on the number of turning points which plays an important role in the computation complexity. The efficiency of \OurMethod algorithm is confirmed by extensive experiments on both synthetic and real data. The experimental results also strongly support our heuristic conjecture on the number of turning points and shed light on the theoretical efficiency of \OurMethod.

\bibliography{HONES}

\begin{thebibliography}{36}
\providecommand{\natexlab}[1]{#1}
\providecommand{\url}[1]{\texttt{#1}}
\expandafter\ifx\csname urlstyle\endcsname\relax
  \providecommand{\doi}[1]{doi: #1}\else
  \providecommand{\doi}{doi: \begingroup \urlstyle{rm}\Url}\fi

\bibitem[Bai and Silverstein(2010)]{bai10}
Zhidong Bai and Jack~W Silverstein.
\newblock \emph{Spectral analysis of large dimensional random matrices},
  volume~20.
\newblock Springer, 2010.

\bibitem[Bank et~al.(1982)Bank, Guddat, Klatte, Kummer, and Tammer]{banknon}
B~Bank, J~Guddat, D~Klatte, B~Kummer, and K~Tammer.
\newblock Non-linear parametric optimization.
\newblock \emph{Akademie-Verlag, Berlin}, 1982.

\bibitem[Beck and Teboulle(2003)]{beck03}
Amir Beck and Marc Teboulle.
\newblock Mirror descent and nonlinear projected subgradient methods for convex
  optimization.
\newblock \emph{Operations Research Letters}, 31\penalty0 (3):\penalty0
  167--175, 2003.

\bibitem[Bertsimas and Tsitsiklis(1997)]{bertsimas1997introduction}
Dimitris Bertsimas and John~N Tsitsiklis.
\newblock \emph{Introduction to linear optimization}, volume~6.
\newblock Athena Scientific Belmont, MA, 1997.

\bibitem[Best(1996)]{best1996algorithm}
Michael~J Best.
\newblock An algorithm for the solution of the parametric quadratic programming
  problem.
\newblock In \emph{Applied mathematics and parallel computing}, pages 57--76.
  Springer, 1996.

\bibitem[Birgin et~al.(2000)Birgin, Mart{\'\i}nez, and
  Raydan]{birgin2000nonmonotone}
Ernesto~G Birgin, Jos{\'e}~Mario Mart{\'\i}nez, and Marcos Raydan.
\newblock Nonmonotone spectral projected gradient methods on convex sets.
\newblock \emph{SIAM Journal on Optimization}, 10\penalty0 (4):\penalty0
  1196--1211, 2000.

\bibitem[Bomze(1998)]{bomze98}
Immanuel~M Bomze.
\newblock On standard quadratic optimization problems.
\newblock \emph{Journal of Global Optimization}, 13\penalty0 (4):\penalty0
  369--387, 1998.

\bibitem[Bomze et~al.(2008)Bomze, Locatelli, and Tardella]{bomze08}
Immanuel~M Bomze, Marco Locatelli, and Fabio Tardella.
\newblock New and old bounds for standard quadratic optimization: dominance,
  equivalence and incomparability.
\newblock \emph{Mathematical Programming}, 115\penalty0 (1):\penalty0 31--64,
  2008.

\bibitem[Brodie et~al.(2009)Brodie, Daubechies, De~Mol, Giannone, and
  Loris]{brodie2009sparse}
Joshua Brodie, Ingrid Daubechies, Christine De~Mol, Domenico Giannone, and
  Ignace Loris.
\newblock Sparse and stable markowitz portfolios.
\newblock \emph{Proceedings of the National Academy of Sciences}, 106\penalty0
  (30):\penalty0 12267--12272, 2009.

\bibitem[Chen and Peng(2015)]{chen2015new}
Xin Chen and Jiming Peng.
\newblock New analysis on sparse solutions to random standard quadratic
  optimization problems and extensions.
\newblock \emph{Mathematics of Operations Research}, 40\penalty0 (3):\penalty0
  725--738, 2015.

\bibitem[Chen et~al.(2013)Chen, Peng, and Zhang]{chen13}
Xin Chen, Jiming Peng, and Shuzhong Zhang.
\newblock Sparse solutions to random standard quadratic optimization problems.
\newblock \emph{Mathematical Programming}, 141\penalty0 (1-2):\penalty0
  273--293, 2013.

\bibitem[Chow et~al.(1979)Chow, Mallet-Paret, and Yorke]{chow79}
Shui-Nee Chow, John Mallet-Paret, and James~A Yorke.
\newblock A homotopy method for locating all zeros of a system of polynomials.
\newblock In \emph{Functional differential equations and approximation of fixed
  points}, pages 77--88. Springer, 1979.

\bibitem[Cover(1991)]{cover1991universal}
Thomas~M Cover.
\newblock Universal portfolios.
\newblock \emph{Mathematical finance}, 1\penalty0 (1):\penalty0 1--29, 1991.

\bibitem[Duchi et~al.(2008)Duchi, Shalev-Shwartz, Singer, and
  Chandra]{duchi2008efficient}
John Duchi, Shai Shalev-Shwartz, Yoram Singer, and Tushar Chandra.
\newblock Efficient projections onto the $l_{1}$-ball for learning in high
  dimensions.
\newblock In \emph{Proceedings of the 25th international conference on Machine
  learning}, pages 272--279. ACM, 2008.

\bibitem[Efron et~al.(2004)Efron, Hastie, Johnstone, Tibshirani,
  et~al.]{efron04}
Bradley Efron, Trevor Hastie, Iain Johnstone, Robert Tibshirani, et~al.
\newblock Least angle regression.
\newblock \emph{The Annals of statistics}, 32\penalty0 (2):\penalty0 407--499,
  2004.

\bibitem[Fan et~al.(2008)Fan, Zhang, and Yu]{fan2008asset}
Jianqing Fan, Jingjin Zhang, and Ke~Yu.
\newblock Asset allocation and risk assessment with gross exposure constraints
  for vast portfolios.
\newblock \emph{Available at SSRN 1307423}, 2008.

\bibitem[Fan et~al.(2012)Fan, Zhang, and Yu]{fan2012vast}
Jianqing Fan, Jingjin Zhang, and Ke~Yu.
\newblock Vast portfolio selection with gross-exposure constraints.
\newblock \emph{Journal of the American Statistical Association}, 107\penalty0
  (498):\penalty0 592--606, 2012.

\bibitem[Frank and Wolfe(1956)]{frank56}
Marguerite Frank and Philip Wolfe.
\newblock An algorithm for quadratic programming.
\newblock \emph{Naval research logistics quarterly}, 3\penalty0 (1-2):\penalty0
  95--110, 1956.

\bibitem[Garrigues and Ghaoui(2009)]{garrigues09}
Pierre Garrigues and Laurent~E Ghaoui.
\newblock An homotopy algorithm for the lasso with online observations.
\newblock In \emph{Advances in neural information processing systems}, pages
  489--496, 2009.

\bibitem[G{\"a}rtner et~al.(2009)G{\"a}rtner, Jaggi, and
  Maria]{gartner2009exponential}
Bernd G{\"a}rtner, Martin Jaggi, and Cl{\'e}ment Maria.
\newblock An exponential lower bound on the complexity of regularization paths.
\newblock \emph{arXiv preprint arXiv:0903.4817}, 2009.

\bibitem[Hazan and Arora(2006)]{hazan2006efficient}
Elad Hazan and Sanjeev Arora.
\newblock \emph{Efficient algorithms for online convex optimization and their
  applications}.
\newblock Princeton University, 2006.

\bibitem[Hazan et~al.(2006{\natexlab{a}})Hazan, Kalai, Kale, and
  Agarwal]{hazan06}
Elad Hazan, Adam Kalai, Satyen Kale, and Amit Agarwal.
\newblock Logarithmic regret algorithms for online convex optimization.
\newblock In \emph{International Conference on Computational Learning Theory},
  pages 499--513. Springer, 2006{\natexlab{a}}.

\bibitem[Hazan et~al.(2006{\natexlab{b}})Hazan, Kalai, Kale, and
  Agarwal]{hazan2006logarithmic}
Elad Hazan, Adam Kalai, Satyen Kale, and Amit Agarwal.
\newblock Logarithmic regret algorithms for online convex optimization.
\newblock In \emph{International Conference on Computational Learning Theory},
  pages 499--513. Springer, 2006{\natexlab{b}}.

\bibitem[Helmbold et~al.(1998)Helmbold, Schapire, Singer, and
  Warmuth]{helmbold98}
David~P Helmbold, Robert~E Schapire, Yoram Singer, and Manfred~K Warmuth.
\newblock On-line portfolio selection using multiplicative updates.
\newblock \emph{Mathematical Finance}, 8\penalty0 (4):\penalty0 325--347, 1998.

\bibitem[Ibaraki and Katoh(1988)]{ibaraki88}
Toshihide Ibaraki and Naoki Katoh.
\newblock \emph{Resource allocation problems: algorithmic approaches}.
\newblock MIT press, 1988.

\bibitem[Kalai and Vempala(2002)]{kalai02}
Adam Kalai and Santosh Vempala.
\newblock Efficient algorithms for universal portfolios.
\newblock \emph{Journal of Machine Learning Research}, 3\penalty0
  (Nov):\penalty0 423--440, 2002.

\bibitem[Kivinen and Warmuth(1997)]{kivinen97}
Jyrki Kivinen and Manfred~K Warmuth.
\newblock Exponentiated gradient versus gradient descent for linear predictors.
\newblock \emph{Information and Computation}, 132\penalty0 (1):\penalty0 1--63,
  1997.

\bibitem[Li(1983)]{li83}
Tien-Yien Li.
\newblock On chow, mallet-paret and yorke homotopy for solving systems of
  polynomials.
\newblock \emph{Bull. Inst. Math. Acad. Sinica}, 11\penalty0 (3):\penalty0
  433--437, 1983.

\bibitem[Mairal and Yu(2012)]{mairal2012complexity}
Julien Mairal and Bin Yu.
\newblock Complexity analysis of the lasso regularization path.
\newblock \emph{arXiv preprint arXiv:1205.0079}, 2012.

\bibitem[Markowitz(1952)]{markowitz1952portfolio}
Harry Markowitz.
\newblock Portfolio selection.
\newblock \emph{The journal of finance}, 7\penalty0 (1):\penalty0 77--91, 1952.

\bibitem[Murty and Yu(1988)]{murty1988linear}
Katta~G Murty and Feng-Tien Yu.
\newblock \emph{Linear complementarity, linear and nonlinear programming}.
\newblock Citeseer, 1988.

\bibitem[Ritter(1981)]{ritter1981parametric}
Klaus Ritter.
\newblock On parametric linear and quadratic programming problems.
\newblock Technical report, DTIC Document, 1981.

\bibitem[Scozzari and Tardella(2008)]{scozzari08}
Andrea Scozzari and Fabio Tardella.
\newblock A clique algorithm for standard quadratic programming.
\newblock \emph{Discrete Applied Mathematics}, 156\penalty0 (13):\penalty0
  2439--2448, 2008.

\bibitem[Shalev-Shwartz(2011)]{shalev2011online}
Shai Shalev-Shwartz.
\newblock Online learning and online convex optimization.
\newblock \emph{Foundations and Trends in Machine Learning}, 4\penalty0
  (2):\penalty0 107--194, 2011.

\bibitem[Wright(1997)]{wright1997primal}
Stephen~J Wright.
\newblock \emph{Primal-Dual Interior-Point Methods}, volume~54.
\newblock SIAM, 1997.

\bibitem[Zinkevich(2003)]{zinkevich03}
Martin Zinkevich.
\newblock Online convex programming and generalized infinitesimal gradient
  ascent.
\newblock In \emph{Proceedings of the 20th International Conference on Machine
  Learning (ICML-03)}, pages 928--936, 2003.

\end{thebibliography}
\bibliographystyle{plainnat}

\appendix
\section{Roadmap of Appendices}
The general idea of \OurMethod algorithm has been presented in section \ref{sec:algo}. However, to implement it efficiently, we need much more effort to explore the structure of the solution path and find common quantities which are used by multiple sub-routines. To make the derivation well-organized, we start from considering the case where only one of $\A$ and $\c$ is time-varying while the other parameter is fixed. The case where $\A$ is time-varying and the case where $\c$ is time-varying are considered separately in Appendix \ref{app:varyA} and \ref{app:varyc}. Then in Appendix \ref{app:varyAc}, we combine two components and state the implementation for the general case. 

In each following appendix, we will first define a list of case-specific intermediate variables, which are the key ingredients to improve efficiency. Then we describe the whole procedure followed by details of each sub-routine. Finally, we give a detailed complexity analysis at the end of each appendix.

\section{Implementation of \OurMethod Algorithm With Time-Varying $\A$ and Fixed $\c$}\label{app:varyA}
\subsection{Intermediate Variables}
Although (\ref{eq:xs})-(\ref{eq:muzero}) completely define the solution, they involves messy terms. To simplify the notations, we define three lists of intermediate variables. We should emphasize that these variables also play important roles in the algorithm design since they capture the quantities repeatedly appeared and unnecessary computation can be avoided by storing their values in memory. 

The intermediate variables are defined as follows. First, let $M$ be a $n\times n$ matrix such that
\begin{equation}
  \label{eq:param_mat}
  \Mss = \Ass^{-1}, \quad \Mcs = -\Acs\Ass^{-1}, \quad \Mc = 0.
\end{equation}
For large-scale problem where $n$ is prohibitively large, we can only store a $n\times |S|$ matrix by removing the zero entries of $M$. This saves storage cost significantly. Then we define two vectors $\eta, \teta\in\R^{n}$ such that 
\begin{equation}
  \label{eq:param_vec}
\etas = \Mss \gs, \quad \etac = \gc + \Mcs\gs, \quad \tetas = \Mss\ones, \quad \tetac = \onec + \Mcs\ones\in\R^{n}.
\end{equation}
Last we define four scalars.
\begin{align}
  &\D = \ones^{T}\Ass^{-1}\ones, \quad \Dg = \ones^{T}\Ass^{-1}\gs, \quad \Dgg = \gs^{T}\Ass^{-1}\gs \quad \Dgy = - \etas^{T}\cs   \label{eq:param_scalar}.
\end{align}
Note that all variables are functions of $\lambda$ if $A$ is replaced by $A + \lambda gg^{T}$ and we denote them by $\cdot(\lambda)$. For example, 
\[\D(\lambda) = \ones^{T}(\Ass + \lambda\gs\gs^{T})^{-1}\ones,\]
and others can be defined in a similar fashion. The following lemma formulates these functions.
\begin{lemma}\label{lem:var_lambda}
Let $\flam = \frac{\lam}{1 + \lam\Dgg}$. Before any entry of $(\xs, \muc)$ hitting $0$, it holds that 
\begin{itemize}
\item $\Ms(\lam) = \Ms - \flam \eta\etas^{T}$;
\item $\eta(\lam) = \frac{\eta}{1 + \lam\Dgg}$;
\item $\teta(\lam) = \teta - \flam \Dg \eta$;
\item $\D(\lam) = \D - \flam \Dg^{2}$;
\item $(\Dg(\lam), \Dgg(\lam), \Dgy(\lam)) = \frac{1}{1 + \lam \Dgg}(\Dg, \Dgg, \Dgy)$.
\end{itemize}
\end{lemma}
\begin{proof}
By Sherman-Morrison-Woodbury formula, 
\[\Mss(\lam) = (\Ass + \lam\gs\gs^{T})^{-1} = \Ass^{-1} - \lam\frac{\Ass^{-1}\gs\gs^{T}\Ass^{-1}}{1 + \lam \gs^{T}\Ass^{-1}\gs} = \Mss - \flam \etas\etas^{T}.\]
This implies that 
\begin{align*}
  \Mcs(\lam) &= - (\Acs + \lam\gc\gs^{T})(\Ass^{-1} - \flam \etas\etas^{T})\\
& = \Mcs - \lam\gc\gs^{T}\Ass^{-1} + \lam\flam \gc\gs^{T}\etas\etas^{T} + \flam\Acs\etas\etas^{T}\\
& = \Mcs - \lam\gc\etas^{T} + \lam\flam \Dgg \gc\etas^{T} + \flam\Acs\etas\etas^{T}\use{\Dgg = \gs^{T}\etas}\\
& = \Mcs - (\lam - \lam\flam\Dgg)\gc\etas^{T} +\flam\Acs\etas\etas^{T}\\
& = \Mcs - \flam \gc\etas^{T} + \flam \Acs \etas\etas^{T}\use{\lam - \lam\flam\Dgg = \flam}\\
& = \Mcs - \flam(\gc - \Acs\Ass^{-1}\gs)\etas^{T}\\
& = \Mcs - \flam \etac\etas^{T}.
\end{align*}
Putting pieces together, we obtain that
\[\Ms(\lam) = \Ms - \flam \eta\etas^{T}.\]
Based on $\Ms(\lam)$, it is straightforward to derive other variables. For $\eta(\lam)$,
\begin{align*}
  \etas(\lam) &= \Mss(\lam)\gs = \etas - \flam\etas\etas^{T}\gs\\
& = (1 - \flam \Dgg)\etas = \frac{\etas}{1 + \lam\Dgg};\\
\etac(\lam) & = \gc + \Mcs(\lam)\gs = \etac - \flam\etac\etas^{T}\gs\\
& = (1 - \flam \Dgg)\etac = \frac{\etac}{1 + \lam\Dgg}.
\end{align*}
Thus,
\[\eta(\lam) = \frac{\eta}{1 + \lam\Dgg}.\]
Similarly,,
\begin{align*}
  \tetas(\lam) &= \Mss(\lam)\ones = \tetas - \flam\etas\etas^{T}\ones\\
& = \tetas - \flam \Dg \etas;\\
\tetac(\lam) & = \onec + \Mcs(\lam)\ones = \onec - \flam\etac\etas^{T}\ones\\
& = \tetac - \flam \Dg \etac,
\end{align*}
and hence 
\[\teta(\lam) = \tetas - \flam\Dg \eta.\]
The last four scalars are even easier to handle. In fact, $\D(\lam)$ can be derived directly by
\[\D(\lam) = \ones^{T}\lb\Mss - \flam \etas\etas^{T}\rb\ones = \D - \flam \Dg^{2}\]
By reformulating the other three variables, the last statement can be proved,
\begin{align*}
(\Dg(\lam), \Dgg(\lam), \Dgy(\lam)) &= (\ones^{T}\etas(\lam), \gs^{T}\etas(\lam),  - \cs^{T}\etas(\lam))\\ 
&= \frac{1}{1 + \lam\Dgg}(\ones^{T}\etas, \gs^{T}\etas,  - \cs^{T}\etas)\\ 
& = \frac{1}{1 + \lam\Dgg}(\Dg, \Dgg, \Dgy).
\end{align*}
\end{proof}

\subsection{Implementation}
Lemma \ref{lem:var_lambda} implies that given the function values of the intermediate variables at $\lam = 0$, the function values at a neighborhood of $0$ can be calculated directly. Within time $t$, all intermediate variables will be update correspondingly when the support changes. It has been shown in Theorem \ref{thm:var_lambda} that updating $\M$ requires $n|S|$ operations while updating other variables only requires $n$ operations. When the problem transfer from time $t$ to time $t + 1$, the variables $(\eta, \Dg, \Dgg, \Dgy)$ needs to be recalculated since it depends on a new $g^{(t + 1)}$. In contrast, $(\M, \teta, \D)$ can be updated in the same way as in time $t$. In summary, $(\M, \teta, \D)$ is shared by for all times while $(\eta, \Dg, \Dgg, \Dgy)$ is only used in a single time. For compact notations, we define $\Par_{1}$ and $\Par_{2}$ as
\begin{equation}
  \label{eq:Par1Par2}
  \Par_{1} = \{\M, \teta, \D\}, \quad \Par_{2} = \{\eta, \Dg, \Dgg, \Dgy\}.
\end{equation}
In addition, we denote $\v$ by the concatenation of $\xs$ and $-\muc$, i.e. 
\begin{equation}\label{eq:v}
\v_{\S} = \xs, \quad \v_{\S^{c}} = -\muc,
\end{equation} 
as a $n\times 1$ vector. It will be shown in the next subsection that $\v(\lam)$ can be expressed in a concise way.

Algorithm \ref{algo:real_time_hop} describes the full implementation of \OurMethod algorithm, which solves the online problem (\ref{eq:stqp}) with $\c$ fixed. The sub-routines involved will be discussed separately in following subsections. Roughly speaking, after initialization, we enter into the outer-loop and try to solve (\ref{algo:real_time_hop}) at time $t$ using the information from time $t - 1$. Starting from $\lam = 0$, we search for the next $\lam$ that pushes one entry of $\v$ to zero. FIND\_LAMBDA fulfills this goal and also reports the corresponding entry $j$. If $j\in \S$ then $j$ is removed from $S$ and otherwise $j$ is added into $\S$. Since $(\v, \mu_{0}, \Par_{1}, \Par_{2})$ are all functions of $\lam$, we update them by UPDATE\_BY\_LAMBDA, in which $\lam^{\mathrm{inc}}$ denotes the increment to reach the next turning point from the current one. Unlike $(\v, \mu_{0})$, $(\Par_{1}, \Par_{2})$ has discontinuity at each turning point $\lam$ due to the change of support $\S$. They are updated by UPDATE\_SHRINK\_SUPPORT and UPDATE\_EXPAND\_SUPPORT depending on whether $\S$ is shrinked or expanded. The procedure is repeated until $\lam$ cross over 1 and an inner-loop finishes. At the end, $\Par_{2}$ is recomputed for new $\g^{(t + 1)}$, which is achieved by DIRECT\_UPDATE. 

\begin{algorithm}
\caption{\OurMethod Algorithm for time-varying $\A$ and fixed $\c$}\label{algo:real_time_hop}
\textbf{Inputs: } Initial matrix $\A^{(0)}$, vectors $\c$, matrix-update-vectors $\{g^{(t)}, t = 1, 2, \ldots\}$.

\textbf{Initialization: }
\begin{algorithmic}
  \STATE $x\gets$ as the optimum corresponding to $A^{(0)}$;
  \STATE $\S\gets \supp(x)$;
  \STATE Calculate $(\x, \mu, \mu_{0})$ via (\ref{eq:xs})-(\ref{eq:muzero})
  \STATE $\v_{\S} \gets \xs, \v_{\S^{c}}\gets -\muc$;
  \STATE Calculate intermediate variables $(\Par_{1}, \Par_{2})$ via (\ref{eq:param_mat})-(\ref{eq:param_scalar}) with $g = \g^{(1)}$.
\end{algorithmic}

\textbf{Procedure: }
\begin{algorithmic}[1]
  \FOR{$t = 1, 2, \cdots$.}
  \STATE $\lambda\gets 0$;
  \WHILE{$\lambda < 1$}
     \STATE $(\lam^{\mathrm{inc}}, j, S^{\mathrm{new}})\gets \mathrm{FIND\_LAMBDA}(S, \v, \mu_{0}; \Par_{1}, \Par_{2})$;
     \STATE $\lam^{\mathrm{inc}} \gets \min\{\lam^{\mathrm{inc}}, 1 - \lam\}$;
     \STATE $\lam\gets \lam + \lam^{\mathrm{inc}}$;
     \STATE $(\v, \mu_{0}; \Par_{1}, \Par_{2})\gets \mathrm{UPDATE\_BY\_LAMBDA}(\lam^{\mathrm{inc}}; \v, \mu_{0}; \Par_{1}, \Par_{2})$;
     \IF {$S^{\mathrm{new}} = S\cup \{j\}$}
        \STATE $(\Par_{1}, \Par_{2})\gets \mathrm{UPDATE\_EXPAND\_SUPPORT}(\lam, S, j; \c, \g^{(t)}, \Par_{1}, \Par_{2})$;
     \ELSIF {$S^{\mathrm{new}} = S\setminus \{j\}$}
        \STATE $(\Par_{1}, \Par_{2})\gets \mathrm{UPDATE\_SHRINK\_SUPPORT}(S, j; \c, \g^{(t)}, \Par_{1}, \Par_{2})$;
     \ENDIF
     \STATE $S\gets S^{\mathrm{new}}$.
  \ENDWHILE
  \STATE $\Par_{2}\gets \mathrm{DIRECT\_UPDATE}(S, \c, g^{(t + 1)}; \Par_{1}, \Par_{2})$;
  \STATE $A\gets A + g^{(t)}(g^{(t)})^{T}$;
  \STATE $x_{\S}^{(t)}\gets \xs, \quad x_{\S^{c}}^{(t)} \gets 0$.
  \ENDFOR
\end{algorithmic}

\textbf{Output:} $x^{(1)}, x^{(2)}, \cdots$.
\end{algorithm}

\subsection{\textnormal{FIND\_LAMBDA}}

With the help of intermediate variables, we can express $(\xs, \muc, \mu_{0})$ in a compact way. 
\begin{theorem}\label{thm:var_lambda}
 Before any entry of $(\xs, \muc)$ hitting $0$, it holds that
\begin{equation}
  \label{eq:mu0}
  \mu_{0}(\lam) = \mu_{0} + \frac{\flam}{\D - \flam \Dg^{2}}\cdot \Dg(\Dg\mu_{0} - \Dgy).
\end{equation}
and 
\begin{equation}
  \label{eq:xmu}
  \v(\lam) \triangleq \com{\xs(\lam)}{-\muc(\lam)} = \v + \frac{\flam}{\D - \flam \Dg^{2}}\cdot (\Dg\mu_{0} - \Dgy)\cdot (\Dg \teta - \D \eta),
\end{equation}
\end{theorem}
\begin{proof}
First we prove (\ref{eq:mu0}). By definition, 
\begin{align*}
\mu_{0}(\lam) &= \frac{1 - \ones^{T}\Ass(\lam)^{-1}\cs}{\ones^{T}\Ass(\lam)^{-1}\ones} = \frac{1 - \ones^{T}\Mss(\lam)\cs}{\D(\lam)}.
\end{align*}
By Lemma \ref{lem:var_lambda}, 
\begin{align*}
&- \Mss(\lam)\cs  = - (\Mss - \flam\etas\etas^{T})\cs\\
= & - \Mss\cs - \flam \etas \etas^{T}\cs = - \Mss\cs + \flam\Dgy\etas.
\end{align*}
Thus the numerator of $\mu_{0}(\lam)$ can be written as
\[1 - \ones^{T}\ys + \ones^{T}(\Mcs^{T}\yc - \Mss\cs)- \flam\Dgy\ones^{T}\etas = \D\mu_{0} - \flam \Dgy\Dg.\]
The denominator of $\mu_{0}(\lam)$, by Lemma \ref{lem:var_lambda}, is formulated as
\[\D(\lam) = \D - \flam \Dg^{2}.\]
Putting the pieces together results in
\[\mu_{0}(\lam) = \frac{\D\mu_{0} - \flam \Dgy\Dg}{\D - \flam\Dg^{2}} = \mu_{0} + \frac{\flam}{\D - \flam \Dg^{2}}\cdot \Dg(\Dg\mu_{0} - \Dgy).\]
Plug $\mu_{0}(\lam)$ into (\ref{eq:xs}), we obtain that
\begin{align*}
  \xs(\lam) &= \mu_{0}(\lam)\tetas(\lam) + \Ass(\lam)^{-1}\cs\\
  & = \Ass^{-1}\cs + \flam\Dgy\etas + (\mu_{0}(\lam) - \mu_{0})\tetas(\lam) + \mu_{0}\tetas(\lam)\\
  & = \xs + \flam\Dgy\etas + (\mu_{0}(\lam) - \mu_{0})\tetas(\lam) + \mu_{0}(\tetas(\lam) - \tetas)\\
 & = \xs + \flam\Dgy\etas + \frac{\flam}{\D - \flam \Dg^{2}}\cdot \Dg(\Dg\mu_{0} - \Dgy)\tetas(\lam) - \mu_{0}\flam \Dg\etas \use{\mathrm{Lemma \ref{lem:var_lambda}}}\\
 & = \xs + \frac{\flam}{\D - \flam \Dg^{2}}\cdot \Dg(\Dg\mu_{0} - \Dgy)\tetas(\lam) - \flam (\Dg\mu_{0} - \Dgy)\etas\\
& = \xs + \frac{\flam}{\D - \flam\Dg^{2}}\cdot (\Dg\mu_{0} - \Dgy)\cdot (\Dg\tetas(\lam) - (\D - \flam\Dg^{2})\etas)\\
& = \xs + \frac{\flam}{\D - \flam\Dg^{2}}\cdot (\Dg\mu_{0} - \Dgy)\cdot \lb\Dg \tetas - \flam\Dg^{2}\etas - (\D - \flam\Dg^{2})\etas\rb\\
& =  \xs + \frac{\flam}{\D - \flam\Dg^{2}}\cdot (\Dg\mu_{0} - \Dgy)\cdot \lb\Dg \tetas - \D\etas\rb.
\end{align*}
Similarly, it follows from (\ref{eq:muc}) that 
\begin{align*}
  -\muc(\lam) &= \mu_{0}(\lam)\tetac(\lam) + \cc + \Mcs(\lam)\cs\\
& = -\muc -\mu_{0}\tetac + (\Mcs(\lam) - \Mcs)\cs + \mu_{0}(\lam)\tetac(\lam)\\
& = -\muc -\mu_{0}\tetac + \flam\Dgy\etac + \mu_{0}(\lam)\tetac(\lam)\\
& = -\muc + \flam\Dgy\etac + (\mu_{0}(\lam) - \mu_{0})\tetac(\lam) + \mu_{0}(\tetac(\lam) - \tetac)\\
& = -\muc + \flam\Dgy\etac + \frac{\flam}{\D - \flam \Dg^{2}}\cdot \Dg(\Dg\mu_{0} - \Dgy)\tetac(\lam) - \mu_{0}\flam \Dg\etac\\
& = -\muc + \frac{\flam}{\D - \flam \Dg^{2}}\cdot \Dg(\Dg\mu_{0} - \Dgy)\tetac(\lam) - \flam (\Dg\mu_{0} - \Dgy)\etac\\
& = -\muc + \frac{\flam}{\D - \flam \Dg^{2}}\cdot (\Dg\mu_{0} - \Dgy)(\Dg\tetac(\lam) - (\D - \flam \Dg^{2})\etac)\\
& = -\muc + \frac{\flam}{\D - \flam \Dg^{2}}\cdot (\Dg\mu_{0} - \Dgy)(\Dg\tetac - \flam \Dg^{2} \etac - (\D - \flam \Dg^{2})\etac)\\
& = -\muc + \frac{\flam}{\D - \flam \Dg^{2}}\cdot (\Dg\mu_{0} - \Dgy)(\Dg\tetac - \D \etac)
\end{align*}
In sum, 
\[  \com{\xs(\lam)}{-\muc(\lam)} = \com{\xs}{-\muc} + \frac{\flam}{\D - \flam \Dg^{2}}\cdot (\Dg\mu_{0} - \Dgy)\cdot (\Dg \teta - \D \eta).\]
\end{proof}

Theorem \ref{thm:var_lambda} indicates that searching for next $\lam$ is equivalent to solve $n$ linear equations. In fact, (\ref{eq:xmu}) can be abbreviated as 
\[\v(\lam) = \v + \frac{\flam}{\D - \flam \Dg^{2}}u = \frac{\D\v - (\Dg^{2}\v - u)\flam}{\D - \flam\Dg^{2}},\]
for $u = (\Dg\mu_{0} - \Dgy)\cdot (\Dg \teta - \D \eta)$. Let
\[\alpha = \mathrm{min}_{+}\left\{\frac{\D \v_{i}}{\Dg^{2}\v_{i} - u_{i}}: i = 1, 2, \ldots, n\right\}\]
where $\mathrm{min}_{+}(\Omega)$ denotes the minimum of all positive numbers contained in set $\Omega$. Then the target $\lam$ is the solution of $\flam = \alpha$, i.e.
\[\lam = \frac{\alpha}{1 - \alpha \Dgg}.\]
We should emphasize that the right-handed side might be negative if $\alpha \Dgg \ge 1$ in which case $\v$ never hits 0. Thus, we should set $\lam$ to be infinity. The implementation of FIND\_LAMBDA is stated in Algorithm \ref{algo:find_lambda}
\begin{algorithm}
  \caption{FIND\_LAMBDA}\label{algo:find_lambda}
  \textbf{Input: } Support $\S$, iterate $\v = \com{\xs}{-\muc}$, $\mu_{0}$, intermediate variables $\Par_{1}, \Par_{2}$.

\textbf{Procedure: }
\begin{algorithmic}[1]
  \STATE $u\gets (\Dg\mu_{0} - \Dgy)(\Dg\teta - \D\eta)$;
  \STATE $\alpha\gets \mathrm{min}_{+}\left\{\frac{\D \v_{i}}{\Dg^{2}\v_{i} - u_{i}}: i = 1, 2, \ldots, n\right\}$;
  \STATE $j \gets \mathrm{argmin}_{+}\left\{\frac{\D \v_{i}}{\Dg^{2}\v_{i} - u_{i}}: i = 1, 2, \ldots, n\right\}$;
  \IF {$\alpha\Dgg < 1$}
  \STATE $\lam^{\mathrm{inc}}\gets \frac{\alpha}{1 - \alpha \Dgg}$;
  \ELSE 
  \STATE $\lam^{\mathrm{inc}}\gets \infty$; 
  \ENDIF
  \IF {$j\in \S$}
  \STATE $\S^{\mathrm{new}} = \S \setminus \{j\}$;
  \ELSE
  \STATE $\S^{\mathrm{new}} = \S\cup \{j\}$.
  \ENDIF
\end{algorithmic}

\textbf{Output: }$\lam^{\mathrm{inc}}, j, \S^{\mathrm{new}}$.
\end{algorithm}

\subsection{Variables Update}
\subsubsection{\textnormal{UPDATE\_BY\_LAMBDA}}
Once the next $\lam$ has been calculated, all variables can be updated via Lemma \ref{lem:var_lambda} and Theorem \ref{thm:var_lambda}.
\begin{algorithm}
\caption{UPDATE\_BY\_LAMBDA}\label{algo:update_by_lambda}
\textbf{Input: } Increment $\lam^{\mathrm{inc}}$; iterate $\v = \com{\xs}{-\muc}$, $\mu_{0}$; intermediate variables $\Par_{1}, \Par_{2}$.

\textbf{Procedure: }
\begin{algorithmic}[1]
  \STATE $\alpha_{0}\gets \frac{1}{1 + \lam^{\mathrm{inc}}\cdot\Dgg}$;
  \STATE $\alpha\gets \lam^{\mathrm{inc}}\cdot\alpha_{0}$;
  \STATE $\td{\alpha}\gets \frac{\alpha}{\D - \alpha\Dg^{2}}$;
  \STATE $\v\gets \v + \td{\alpha}\cdot (\Dg\mu_{0} - \Dgy)(\Dg\teta - \D\eta)$;
  \STATE $\mu_{0}\gets \mu_{0} + \td{\alpha}\cdot \Dg(\Dg\mu_{0} - \Dgy)$;
  \STATE $\D \gets \D - \alpha \Dg^{2}$;
  \STATE $(\Dg, \Dgg, \Dgy) \gets \alpha_{0} (\Dg, \Dgg, \Dgy)$.
  \STATE $\Ms\gets \Ms - \alpha \eta\etas^{T}$;
  \STATE $\teta\gets \teta - \alpha \Dg\eta$;
  \STATE $\eta\gets \alpha_{0}\eta$;
\end{algorithmic}

\textbf{Output: }$\v, \mu_{0}, \Par_{1}, \Par_{2}$.
\end{algorithm}

\subsubsection{\textnormal{UPDATE\_EXPAND\_SUPPORT}}
Suppose $S$ is updated to $S \cup \{j\}$ for some $j\in S^{c}$. Denote $\T$ by $S\cup \{j\}$ and we add a supscript $+$ to each variable to denote the value after update. The key tool is the following formula showing the relation between matrix inverses after adding one row and one column. 
\begin{proposition}\label{prop:block_add}
Let $\tAjj = \Ajj - \Ajs\Ass^{-1}\Asj$,
 \[\Att^{-1} = \mat{\Ass^{-1}}{0}{0}{0} + \frac{1}{\tAjj}\cdot \com{-\Ass^{-1}\Asj}{1}\rcom{-\Ajs\Ass^{-1}}{1}.\]
\end{proposition}
Similar to section 4.1, the key is to update $\M$ and other variables are easy to update based on $\M$. Denote a class of operator $\{\opr: j\in S^{c}\}$ for matrix $W\in\R^{n\times n}$, $\opr(W)$ sets the $j$-th row and $j$-th column of $W$ to be zero and for vector $z\in\R^{n\times 1}$, $\opr(z)$ sets the $j$-th coordinate of $z$ to be zero. One property of $\opr$ to be used is that For any matrix-vector pair $(W, z)$, 
\begin{equation}
  \label{eq:opr}
  \opr(W)z = \opr(Wz) - z_{j}\opr(W_{j})
\end{equation}
where $W_{j}$ is $j$-th column of $W$.
\begin{theorem}\label{thm:expand_M}
Let $\gamma$ and $\td{\gamma}$ be two $n\times 1$ vectors with
\[\gamma_{\T} = \td{\gamma}_{\T} = \rcom{\Mjs}{1}^{T}, \quad \gamma_{\T^{c}} = -\ACj - \ACs\Mjs^{T}, \quad \td{\gamma}_{\T^{c}} = 0.\]
Then 
  \[\M^{+} = \opr(\M) + \frac{1}{\tAjj}\cdot\gamma\td{\gamma}^{T}.\]
\end{theorem}
\begin{proof}
By definition, 
\[\Mtt^{+} = \Att^{-1}, \quad \MCt^{+} = -\ACt\Att^{-1}.\]  
By Proposition \ref{prop:block_add},
\[\Mtt^{+} = \mat{\Mss}{0}{0}{0} + \frac{1}{\tAjj}\com{\Mjs^{T}}{1}\rcom{\Mjs}{1} = \mat{\Mss}{0}{0}{0} + \frac{1}{\tAjj}\gamma_{\T}\td{\gamma}_{\T}^{T},\]
and 
\begin{align*}
 \MCt^{+} &= -\rcom{\ACs}{\ACj}\left\{\mat{\Mss}{0}{0}{0} + \frac{1}{\tAjj}\com{\Mjs^{T}}{1}\rcom{\Mjs}{1}\right\}\\
& = \rcom{\MCs}{0} + \frac{1}{\tAjj}\gamma_{\T^{c}}\td{\gamma}_{\T}^{T}.
\end{align*}
Note that $\MC$ is always a zero matrix by definition, the above results imply that
\[\M^{+} = \opr(\M) + \frac{1}{\tAjj}\cdot\gamma\td{\gamma}^{T}.\]
\end{proof}

The update of other parameters can be derived as a consequence of Theorem \ref{thm:expand_M}. Theorem \ref{thm:expand_other} summarizes the result.

\begin{theorem}\label{thm:expand_other}
Let $b_{j, \S} = - \c_{\T}^{T}\gamma_{\T}$, then
  \begin{itemize}
  \item $\eta^{+} = \opr(\eta) + \frac{\eta_{j}}{\tAjj}\gamma$;
  \item $\teta^{+} = \opr(\teta) + \frac{\td{\eta}_{j}}{\tAjj}\gamma$;
  \item $\D^{+} = \D + \frac{\td{\eta}_{j}^{2}}{\tAjj}$
  \item $\Dg^{+} = \Dg + \frac{\eta_{j}\td{\eta}_{j}}{\tAjj}$;
  \item $\Dgg^{+} = \Dgg + \frac{\eta_{j}^{2}}{\tAjj}$;
  \item $\Dgy^{+} = \Dgy + \frac{\eta_{j}b_{j, \S}}{\tAjj}$;
  \end{itemize}
\end{theorem}
\begin{proof}
 Since $\M_{j} = 0$, (\ref{eq:opr}) implies that for any $z\in\R^{n\times 1}$
\[\opr(\M z) = \opr(\M)z.\]
By definition,
\[\eta = \com{0}{\gc} + \M \g, \quad \teta = \com{0}{\onec} + \M\one.\]
Also notice that $\td{\gamma}_{\T}^{T}\gt = \g_{j} + \Mjs^{T}\gs = \eta_{j}$ and $\gamma_{\T}^{T}\onet = 1 + \Mjs\ones = \teta_{j}$, thus, 
\begin{align*}
\eta^{+} &= \com{0}{\gC} + \M^{+}g = \com{0}{\gC} + \opr(M) \g + \frac{\td{\gamma}^{T}\g}{\tAjj}\gamma\\
& = \com{0}{\gC} + \opr(M\g) + \frac{\td{\gamma}_{\T}^{T}\gt}{\tAjj}\gamma\\
& = \com{0}{\gC} + \opr(\eta) - \opr\lb\com{0}{\gc}\rb + \frac{\eta_{j}}{\tAjj}\gamma\\
& = \com{0}{\gC} + \opr(\eta) - \com{0}{\gC} + \frac{\eta_{j}}{\tAjj}\gamma\\
& = \opr(\eta) + \frac{\eta_{j}}{\tAjj}\gamma.
\end{align*}
The update of $\teta$ can be obtained by replacing $g$ by $\one$ in the above derivation. The four scalars $\D, \Dg, \Dgg, \Dgy$ can be updated as follows.
\begin{align*}
\D^{+} &= \onet^{T}\teta_{\T}^{+} = \onet^{T}\lb\opr(\teta)_{\T} + \frac{\td{\eta}_{j}}{\tAjj}\gamma_{\T}\rb = \D + \frac{\td{\eta}_{j}^{2}}{\tAjj};\\
\Dg^{+} &= \onet^{T}\eta_{\T}^{+} = \onet^{T}\lb\opr(\eta)_{\T} + \frac{\eta_{j}}{\tAjj}\gamma_{\T}\rb = \Dg + \frac{\td{\eta}_{j}\eta_{j}}{\tAjj};\\
\Dgg^{+} &= \gt^{T}\eta_{\T}^{+} = \gt^{T}\lb\opr(\eta)_{\T} + \frac{\eta_{j}}{\tAjj}\gamma_{\T}\rb = \Dgg + \frac{\eta_{j}^{2}}{\tAjj};\\
\Dgy^{+} &=  - \c_{\T}^{T}\eta_{\T}^{+} =  - \c_{\T}^{T}\lb\opr(\eta_{\T}) + \frac{\eta_{j}}{\tAjj}\gamma_{\T}\rb = \Dgy - \frac{\eta_{j}}{\tAjj}\c_{\T}^{T}\gamma_{\T}\\
& = \Dgy + \frac{\eta_{j}b_{j, \S}}{\tAjj}.
\end{align*}
\end{proof}

The implementation of UPDATE\_EXPAND\_SUPPORT is summarized in Algorithm \ref{algo:update_expand_support}. Note that both $\tAjj$ and $\gamma_{\T^{c}}$ depends on $\lam$ and it is easy to see that 
\begin{align*}
  &\tAjj(\lam) = \Ajj + \lam \g_{j}^{2} + \Mjs(\Asj + \lam \g_{j}\gs) = \Ajj + \Mjs\Asj + \lam \g_{j}\eta_{j}\\
  & \gamma_{\T^{c}}(\lam)\gets - (\A_{\T^{c}j} + \lam \g_{\T^{c}}\g_{j}) - (\A_{\T^{c}\S} + \lam \g_{\T^{c}}\gs^{T})\Mjs^{T} = -\A_{\T^{c}j} - \A_{\T^{c}\S}\Mjs^{T} - \lam\eta_{j}\g_{\T^{c}}.
\end{align*}

\begin{algorithm}
  \caption{UPDATE\_EXPAND\_SUPPORT}\label{algo:update_expand_support}
  \textbf{Inputs: } Current $\lam$, original support $S$, new index $j$, matrix $A$, vectors $\y, \c, \g$, intermediate variables $\Par_{1}, \Par_{2}$.

\textbf{Procedure: }
\begin{algorithmic}[1]
  \STATE $\tAjj\gets \Ajj + \Mjs\Asj + \lam \g_{j}\eta_{j}$;
  \STATE $\gamma_{\T}\gets (\Mjs, 1)^{T}, \gamma_{\T^{c}}\gets -\A_{\T^{c}j} - \A_{\T^{c}\S}\Mjs^{T} - \lam\eta_{j}\g_{\T^{c}}$;
  \STATE $\td{\gamma}_{\T}\gets (\Mjs, 1)^{T}, \td{\gamma}_{\T^{c}}\gets 0$;
  \STATE $b\gets  - \c_{\T}^{T}\gamma_{\T}$;
  \STATE $\D\gets \D + \frac{\teta_{j}^{2}}{\tAjj}$;
  \STATE $\Dg\gets \Dg + \frac{\eta_{j}\teta_{j}}{\tAjj}$;
  \STATE $\Dgg\gets \Dgg + \frac{\eta_{j}^{2}}{\tAjj}$;
  \STATE $\Dgy\gets \Dgy + \frac{\eta_{j}b}{\tAjj}$;
  \STATE $\M_{\bigcdot, \T}\gets \opr(\M_{\bigcdot, \T}) + \frac{1}{\tAjj}\gamma\td{\gamma}_{\T}^{T}$;
  \STATE $\eta\gets \opr(\eta) + \frac{\eta_{j}}{\tAjj}\gamma$.
  \STATE $\teta\gets \opr(\teta) + \frac{\teta_{j}}{\tAjj}\gamma$;
\end{algorithmic}

\textbf{Output: }$\Par_{1}, \Par_{2}$.
\end{algorithm}

\subsubsection{\textnormal{UPDATE\_SHRINK\_SUPPORT}}
Suppose $S$ is updated to $S \setminus \{j\}$ for some $j\in S^{c}$. Denote $\T$ by $S\setminus \{j\}$ and we add a supscript $-$ to each variable to denote the value after update. Similar to last subsection, we start from deriving $\M^{-}$ and apply the result to calculate other variables.

\begin{theorem}\label{thm:shrink_M}
Let $\beta$ and $\td{\beta}$ be two $n\times 1$ vectors with
\[\beta_{\T} = \td{\beta}_{\T} = \Mtj, \quad \beta_{\T^{c}} = \com{-1}{\Mcj}\quad \td{\beta}_{\T^{c}} = 0.\]
Then 
  \[\M^{-} = \opr(\M) - \frac{1}{\Mjj}\cdot\beta\td{\beta}^{T}.\]
\end{theorem}

\begin{proof}
  By definition,
\[\mat{\Mtt}{\Mtj}{\Mjt}{\Mjj} = \Ass^{-1} = \mat{\Att}{\Atj}{\Ajt}{\Ajj}^{-1}.\]
Then Proposition \ref{prop:block_add} implies that 
\begin{equation}\label{eq:Mtt}
\mat{\Mtt}{\Mtj}{\Mjt}{\Mjj} = \mat{\Att^{-1}}{0}{0}{0} + \frac{1}{\Ajj - \Ajt\Att^{-1}\Atj}\cdot \com{-\Att^{-1}\Atj}{1}\rcom{-\Ajt\Att^{-1}}{1}.
\end{equation}
This entails that 
\begin{equation}\label{eq:first_row}
\Att^{-1} = \Mtt - \frac{\Mtj\Mjt}{\Mjj}, \quad -\Ajt\Att^{-1} = \frac{\Mjt}{\Mjj}.
\end{equation}
On the other hand,
\begin{align}
\rcom{\Mct}{\Mcj} &= -\Acs\Ass^{-1} = -\rcom{\Act}{\Acj}\mat{\Mtt}{\Mtj}{\Mjt}{\Mjj}\nonumber\\ 
&= -\rcom{\Act\Mtt + \Acj\Mjt}{\Act\Mtj + \Acj\Mjj}.\label{eq:last_row}
\end{align}
If follows from (\ref{eq:Mtt}), \eqref{eq:first_row} and (\ref{eq:last_row}) that 
\begin{align}
-\Act\Att^{-1} &= -\Act\lb\Mtt - \frac{\Mtj\Mjt}{\Mjj}\rb\nonumber\\
& = \Mct + \Acj\Mjt + \frac{\Act\Mtj\Mjt}{\Mjj}\nonumber\\
& = \Mct + (\Acj \Mjj + \Act \Mtj)\frac{\Mjt}{\Mjj}\nonumber\\
& = \Mct - \frac{\Mcj\Mjt}{\Mjj}.
\end{align}
Putting (\ref{eq:first_row}) and (\ref{eq:last_row}) together, we obtain that
\[\M_{\bigcdot, \T}^{-} = \com{\Att^{-1}}{-\ACt\Att^{-1}} = \lb
\begin{array}{c}
  \Att^{-1} \\ -\Ajt\Att^{-1} \\ -\Act\Att^{-1}
\end{array}
\rb = \opr(\M)_{\bigcdot, \T} - \frac{1}{\Mjj}\cdot\beta\td{\beta}_{\T}^{T}.\]
Since $\M_{\bigcdot, \T^{c}}$ is a zero matrix, 
\[\M^{-} = \opr(\M) - \frac{1}{\Mjj}\cdot\beta\td{\beta}^{T}.\]
\end{proof}

\begin{algorithm}
  \caption{UPDATE\_SHRINK\_SUPPORT}\label{algo:update_shrink_support}
  \textbf{Inputs: } Original support $S$, new index $j$, matrix $A$, vector $\y, \c, \g$, intermediate variables $\Par_{1}, \Par_{2}$.

\textbf{Procedure: }
\begin{algorithmic}[1]
  \STATE $\beta_{\T}\gets \Mjt^{T}, \beta_{\T^{c}}\gets \com{-1}{\MCj}, \td{\beta}_{\T}\gets \Mjt^{T}, \td{\beta}_{\T^{c}}\gets 0$;
  \STATE $\td{b}\gets  - \c_{\T}^{T}\beta_{\T} - \c_{j}\Mjj$;
  \STATE $\D\gets \D - \frac{\teta_{j}^{2}}{\Mjj}$;
  \STATE $\Dg\gets \Dg - \frac{\eta_{j}\teta_{j}}{\Mjj}$;
  \STATE $\Dgg\gets \Dgg - \frac{\eta_{j}^{2}}{\Mjj}$;
  \STATE $\Dgy\gets \Dgy - \frac{\eta_{j}\td{b}}{\Mjj}$;
  \STATE $\M_{\bigcdot, \T}\gets \opr(\M_{\bigcdot, \T}) - \frac{1}{\Mjj}\beta\td{\beta}_{\T}^{T}, \,\,\,\M_{\bigcdot, j}\gets 0$;
  \STATE $\eta\gets \opr(\eta) - \frac{\eta_{j}}{\Mjj}\beta$;
  \STATE $\teta\gets \opr(\td{\eta}) - \frac{\teta_{j}}{\Mjj}\beta$.
\end{algorithmic}

\textbf{Output: }$\Par_{1}, \Par_{2}$.
\end{algorithm}

\begin{theorem}\label{thm:shrink_other}
Let $\td{b}_{j, \S} = - \c_{\T}^{T}\beta_{\T} - \c_{j}\Mjj$, then
  \begin{itemize}
  \item $\eta^{-} = \opr(\eta) - \frac{\eta_{j}}{\Mjj}\beta$;
  \item $\teta^{-} = \opr(\teta) - \frac{\teta_{j}}{\Mjj}\beta$;
  \item $\D^{-} = \D - \frac{\teta_{j}^{2}}{\Mjj}$;
  \item $\Dg^{-} = \Dg - \frac{\eta_{j}\teta_{j}}{\Mjj}$;
  \item $\Dgg^{-} = \Dgg - \frac{\eta_{j}^{2}}{\Mjj}$;
  \item $\Dgy^{-} = \Dgy - \frac{\eta_{j}\td{b}_{j, \S}}{\Mjj}$.
  \end{itemize}
\end{theorem}
\begin{proof}
By (\ref{eq:opr}), 
\[\opr(\M)\g = \opr(\M\g) - \g_{j}\opr(\M_{\bigcdot, j})\]
Let $e_{j}$ is the $j$-th basis vector with $j$-th entry equal to 1 and all other entries equal to 0. Then
\begin{align*}
\eta^{-} &= \com{0}{\gC} + \M^{-}g = \com{0}{\gC} + \opr(M) \g - \frac{\td{\beta}^{T}\g}{\Mjj}\beta\\
& = \com{0}{\gC} + \opr(M\g) - \g_{j}\opr(\M_{\bigcdot, j}) - \frac{\td{\beta}_{\T}^{T}\gt}{\Mjj}\beta\\
& = \com{0}{\gC} + \opr(M\g) - \g_{j}e_{j} - \g_{j}\beta - \frac{\td{\beta}_{\T}^{T}\gt}{\Mjj}\beta\\
& = \com{0}{\gC} + \opr(M\g) - \g_{j}e_{j} - \frac{\Mjs\gs}{\Mjj}\beta\\
& = \com{0}{\gC} + \opr(\eta) - \opr\lb\com{0}{\gc}\rb - \g_{j}e_{j}- \frac{\eta_{j}}{\Mjj}\beta\\
& = \com{0}{\gC} + \opr(\eta) - \com{0}{\gc} - \g_{j}e_{j} - \frac{\eta_{j}}{\Mjj}\beta\\
& = \opr(\eta) - \frac{\eta_{j}}{\Mjj}\beta.
\end{align*}
 Substitute $g$ by $\one$, we obtain the update for $\teta$. Together with (\ref{eq:first_row}) and the fact that $j\in \S$, we obtain that
\begin{align*}
\D^{-} &= \onet^{T}\teta_{\T}^{-} = \onet^{T}\lb\teta_{\T} - \frac{\teta_{j}}{\Mjj}\beta_{\T}\rb = \D - \teta_{j} - \frac{\teta_{j}}{\Mjj}(\onet^{T}\Mtj)\\
& = \D - \frac{\teta_{j}}{\Mjj}(\Mjj + \onet^{T}\Mtj) = \D - \frac{\teta_{j}^{2}}{\Mjj};\\
\Dg^{-} &= \onet^{T}\eta_{\T}^{-} = \onet^{T}\lb\eta_{\T} - \frac{\eta_{j}}{\Mjj}\beta_{\T}\rb = \Dg - \eta_{j} - \frac{\eta_{j}}{\Mjj}(\onet^{T}\Mtj)\\
& = \Dg - \frac{\eta_{j}}{\Mjj}(\Mjj + \onet^{T}\Mtj) = \Dg - \frac{\eta_{j}\teta_{j}}{\Mjj};\\
\Dgg^{-} &= \gt^{T}\eta_{\T}^{-} = \gt^{T}\lb\eta_{\T} - \frac{\eta_{j}}{\Mjj}\beta_{\T}\rb = \Dgg - g_{j}\eta_{j} - \frac{\eta_{j}}{\Mjj}(\gt^{T}\Mtj)\\
& = \Dgg - \frac{\eta_{j}}{\Mjj}(g_{j}\Mjj + \gt^{T}\Mtj) = \Dgg - \frac{\eta_{j}^{2}}{\Mjj};\\
\Dgy^{-} &= - \c_{\T}^{T}\eta_{\T}^{-} =  - \c_{\T}^{T}\lb\eta_{\T} - \frac{\eta_{j}}{\Mjj}\beta_{\T}\rb =  - \cs^{T}\etas + \c_{j}\eta_{j} + \frac{\eta_{j}}{\Mjj}\c_{\T}^{T}\beta_{\T}\\
& = \Dgy - \frac{\eta_{j}\td{b}_{j, \S}}{\Mjj}.
\end{align*}
\end{proof}

The implementation of UPDATE\_SHRINK\_SUPPORT is summarized in Algorithm \ref{algo:update_shrink_support}.

\subsubsection{DIRECT\_UPDATE}
At the beginning of each time $t$, we need to recompute $\Par_{2} = \{\eta, \Dg, \Dgg, \Dgy\}$. The implementation is summarized in Algorithm \ref{algo:direct_update}.
\begin{algorithm}
  \caption{DIRECT\_UPDATE}\label{algo:direct_update}
  \textbf{Inputs: } Support $S$, vector $\y, \c, \g$, intermediate variables $\Par_{1}, \Par_{2}$. 

\textbf{Procedure: }

\begin{algorithmic}[1]
  \STATE $\etas \gets \Mss\gs, \etac\gets \gc + \Mcs\gs$;
  \STATE $\Dg\gets \ones^{T}\etas$;
  \STATE $\Dgg\gets \etas^{T}\gs$;
  \STATE $\Dgy\gets  - \etas\cs$.
\end{algorithmic}

\textbf{Output: }$\Par_{2}$.
\end{algorithm}

\subsection{Update of $A$}
As will be shown in next subsection, the complexities of all above sub-routines are at most $O(ns)$ where $s = |S|$. However, the complexity of line 16 is $O(n^{2})$ which might dominate  when the solution is sparse and the number of turning points is small. Fortunately, UPDATE\_EXPAND\_SUPPORT is the only sub-routine which extracts information from $A$. In fact, in line 1 and line 2,
\[\com{\tAjj}{\gamma_{\T^{c}}} = \com{\A_{jj} + \M_{jS}\A_{Sj}}{-\A_{\T^{c}j} - \A_{\T^{c}S}\M_{jS}^{T}} + \lambda\eta_{j}\com{\g_{j}}{\g_{\T^{c}}}.\]
This only requires the $j$-th column of $\A$. Let $\S_{*}$ be the union of all supports appeared in Algorithm \ref{algo:real_time_hop}. Suppose we know $\S_{*}$ apriori, we can only update the columns of $\A$ with indices in $\S_{*}$. In other words, we update $\A_{\bigcdot, \S_{*}}$ by $\A_{\bigcdot, \S_{*}} + \lambda \g\g_{\S_{*}}^{T}$ at the beginning of each step and hence the complexity is reduced to $O(n|\S_{*}|)$.

Although agnostic to $\S_{*}$ in reality, we can initialize it by $\supp(x_{k})$ for some positive $k$, e.g. $k = 1$, and keep track it by adding index into $\S_{*}$ once the index is not included in $\S_{*}$. Once a new index $j$ is detected, we update $j$-th column of $\A$ by using all previous $g^{(t)}$. The implementation is stated in Algorithm \ref{algo:real_time_hop_smart_A}. 
\begin{algorithm}
\caption{\OurMethod Algorithm for time-varying $\A$ and fixed $\c$ with sparse update of $\A$}\label{algo:real_time_hop_smart_A}
\textbf{Inputs: } Initial matrix $\A^{(0)}$, vectors $\c$, matrix-update-vectors $\{g^{(t)}, t = 1, 2, \ldots\}$.

\textbf{Initialization: }
\begin{algorithmic}
  \STATE $x\gets$ as the optimum corresponding to $A^{(0)}$;
  \STATE $\S\gets \supp(x), \S_{*}\gets \S$;
  \STATE Calculate $(\x, \mu, \mu_{0})$ via (\ref{eq:xs})-(\ref{eq:muzero})
  \STATE $v_{\S}\gets \xs, \quad v_{\S^{c}}\gets -\muc$;
  \STATE Calculate intermediate variables $(\Par_{1}, \Par_{2})$ via (\ref{eq:param_mat})-(\ref{eq:param_scalar}) based on $\g^{(1)}$.
\end{algorithmic}

\textbf{Procedure: }
\begin{algorithmic}[1]
  \FOR{$t = 1, 2, \cdots$.}
  \STATE $\lambda\gets 0$;
  \WHILE{$\lambda < 1$}
     \STATE $(\lam^{\mathrm{inc}}, j, S^{\mathrm{new}})\gets \mathrm{FIND\_LAMBDA}(S, \v; \Par_{1}, \Par_{2})$;
     \STATE $\lam^{\mathrm{inc}} \gets \min\{\lam^{\mathrm{inc}}, 1 - \lam\}$;
     \STATE $\lam\gets \lam + \lam^{\mathrm{inc}}$;
     \STATE $(\v, \mu_{0}; \Par_{1}, \Par_{2})\gets \mathrm{UPDATE\_BY\_LAMBDA}(\lam^{\mathrm{inc}}; \v, \mu_{0}; \Par_{1}, \Par_{2})$;
     \IF {$S^{\mathrm{new}} = S\cup \{j\}$}
        \STATE $(\Par_{1}, \Par_{2})\gets \mathrm{UPDATE\_EXPAND\_SUPPORT}(\lam, S, j; \A, \c, \g^{(t)}; \Par_{1}, \Par_{2})$;
        \IF {$j\not\in \S_{*}$}
        \STATE $G\gets (g^{(1)}, \ldots, g^{(t - 1)})$;
        \STATE $\A_{\bigcdot, j}\gets \A_{\bigcdot, j} + GG_{j, \bigcdot}^{T}$;
        \STATE $\S_{*} = \S_{*}\cup \{j\}$;
        \ENDIF
     \ELSIF {$S^{\mathrm{new}} = S\setminus \{j\}$}
        \STATE $(\Par_{1}, \Par_{2})\gets \mathrm{UPDATE\_SHRINK\_SUPPORT}(S, j; \c, \g^{(t)}; \Par_{1}, \Par_{2})$;
     \ENDIF
     \STATE $S\gets S^{\mathrm{new}}$.
  \ENDWHILE
  \STATE $\Par_{2}\gets \mathrm{DIRECT\_UPDATE}(S, \c, g^{(t + 1)}; \Par_{1}, \Par_{2})$;
  \STATE $A_{\bigcdot, \S_{*}}\gets A_{\bigcdot, \S_{*}} + g^{(t)}(g_{\S_{*}}^{(t)})^{T}$;
  \STATE $x_{\S}^{(t)}\gets \xs, \quad x_{\S^{c}}^{(t)} \gets 0$.
  \ENDFOR
\end{algorithmic}

\textbf{Output:} $x^{(1)}, x^{(2)}, \cdots$..
\end{algorithm}

\subsection{Complexity Analysis}
In this subsection, we analyze the complexity of the algorithm. We distinguish four types of computation, namely matrix-vector product, outer-product of two vectors, inner-product of two vectors and vector-scalar product. Denote by $W\in\R^{m\times p}$, $(z, \td{z})\in\R^{p}\times\R^{q}$ and $a\in \R$ the generic matrix, vector and scalar respectively. As a convention, the complexity is defined as the number of scalar-scalar multiplications. The addition is omitted here for simplicity. Note that the complexities of $Wz$, $z\td{z}^{T}$, $z^{T}z$ and $az$ are $mp$, $pq$, $p$ and $p$, respectively. The results for a single step are summarized in Table \ref{tab:complexity} where $s_{*} = |\S_{*}|$ be the size of $\S_{*}$ at the final round. We should emphasize that our complexity analysis is exact. 

\begin{table}[htp]
  \centering
  \caption{Computation complexity of each sub-routine in Algorithm \ref{algo:real_time_hop}.}\label{tab:complexity}
  \begin{tabular}{lllll}
    \toprule
    & ($Wz$)-type & ($z\td{z}^{T}$)-type & ($z^{T}z$)-type & ($az$)-type\\ 
    \midrule
    FIND\_LAMBDA & 0 & 0 & $n$ & $2n$\\
    UPDATE\_BY\_LAMBDA & 0 & $ns$ & 0 & $4n$\\
    UPDATE\_EXPAND\_SUPPORT & $s(n - s - 1)$ & $n(s + 1)$ & $n$ & $2n$\\
    UPDATE\_SHRINK\_SUPPORT & 0 & $n(s - 1)$ & $n$ & $2n$\\
    DIRECT\_UPDATE & $ns$ & 0 & $n + s$ & 0\\
    Update of $A$.   & 0 & $ns_{*}$ & 0 & 0\\
    \bottomrule
  \end{tabular}
\end{table}

For given $t$, denote $k_{\A}^{+}$ by the number of turning points which add element to $\S$ and $k_{\A}^{-}$ by the number of turning points which delete element from $\S$. Let $k_{\A} = k_{\A}^{+} + k_{\A}^{-}$ be the total number of turning points and $s$ be the maximum size of $S$ in the iteration Then the complexity of \OurMethod for a single time $t$ is at most
\[ns(3k_{\A}^{+} + 2k_{\A}^{-}) + n(12k_{\A}^{+} + 10k_{\A}^{-}) + O(k_{\A}),\]
Therefore, the complexity at time $t$ is at most
\[C_{1t} \le ns_{*} + ns(3k_{\A}^{+} + 2k_{\A}^{-} + 1) + n(12k_{\A}^{+} + 10k_{\A}^{-} + 2) + O(k_{\A}) \le ns_{*} + ns(3k_{\A} + 1) + n(12k_{\A} + 2) + O(k_{\A}).\]

\section{Implementation of \OurMethod Algorithm With Time-Varying $\c$ and Fixed $\A$}\label{app:varyc}

\subsection{Intermediate Variables}
Similar to Appendix \ref{app:varyA}, we define $\Par_{1} = \{\M, \teta, \D\}$ where the parameters are defined in \eqref{eq:param_mat}-\eqref{eq:param_scalar}. Moreover, we define a vector $\xi$ such that 
\[\xis = -\Ass^{-1}\ls, \quad \xic = -\lc + \Acs\Ass^{-1}\ls,\]
and a scalar $\Dl$ as 
\[\Dl = \ones^{T}\xis.\]
We write $\Par_{3}$ for $\{\xi, \Dl\}$ for convenience.

\subsection{Implementation}
Algorithm \ref{algo:real_time_hop_ulam} describes the full implementation in this case and the sub-routines will be discussed separately in following subsections.
\begin{algorithm}
\caption{\OurMethod Algorithm for constant $A, \y$ and time-varying $\c$}\label{algo:real_time_hop_ulam}
\textbf{Inputs: } Initial matrix $A$, vector $\c^{(0)}$, vector-update-vector $\{\l^{(t)} = \c^{(t)} - \c^{(t - 1)}: t = 1, 2, \ldots\}$.

\textbf{Initialization: }
\begin{algorithmic}
  \STATE $x\gets$ as the optimum corresponding to $\c^{(0)}$.
  \STATE $S\gets \supp(x)$;
  \STATE Calculate $(\x, \mu, \mu_{0})$ via (\ref{eq:xs})-(\ref{eq:muzero})
  \STATE $v_{\S} \gets \xs, v_{\S^{c}}\gets -\muc$;
  \STATE Calculate intermediate variables $\Par_{1}, \Par_{3}$ via (\ref{eq:param_mat})-(\ref{eq:param_scalar}) based on $\l^{(1)}$.
\end{algorithmic}

\textbf{Procedure: }
\begin{algorithmic}[1]
  \FOR{$t = 1, 2, \cdots$.}
  \STATE $\ulam\gets 0$;
  \WHILE{$\ulam < 1$}
     \STATE $(\ulam^{\mathrm{inc}}, j, S^{\mathrm{new}})\gets \mathrm{FIND\_UTILDE\_LAMBDA}(\S, \v; \Par_{1}, \Par_{3})$;
     \STATE $\ulam^{\mathrm{inc}} \gets \min\{\ulam^{\mathrm{inc}}, 1 - \ulam\}$;
     \STATE $(\v, \mu_{0}; \Par_{1}, \Par_{3})\gets \mathrm{UPDATE\_BY\_UTILDE\_LAMBDA}(\ulam^{\mathrm{inc}}; \v, \mu_{0}, \Par_{1}, \Par_{3})$;
     \IF {$S^{\mathrm{new}} = S\cup \{j\}$}
        \STATE $(\Par_{1}, \Par_{3})\gets \mathrm{UPDATE\_UTILDE\_EXPAND\_SUPPORT}(\S, j, \A, \l; \Par_{1}, \Par_{3})$;
     \ELSIF {$S^{\mathrm{new}} = S\setminus \{j\}$}
        \STATE $(\Par_{1}, \Par_{3})\gets \mathrm{UPDATE\_UTILDE\_SHRINK\_SUPPORT}(\S, j, \l; \Par_{1}, \Par_{3})$;
     \ENDIF
     \STATE $S\gets S^{\mathrm{new}}$;
     \STATE $\ulam\gets \ulam + \ulam^{\mathrm{inc}}$.
  \ENDWHILE
  \STATE $\Par_{3}\gets \mathrm{DIRECT\_UTILDE\_UPDATE}(S, \Par_{1}, \h^{(t + 1)})$;
  \STATE $x_{\S}^{(t)}\gets \xs, \quad x_{\S^{c}}^{(t)} \gets 0$.
  \ENDFOR
\end{algorithmic}

\textbf{Output:} $x^{(1)}, x^{(2)}, \cdots$.
\end{algorithm}

\subsection{\textnormal{FIND\_UTILDE\_LAMBDA}}
Define $\v$ as in \eqref{eq:v}. Then Theorem \ref{thm:update} implies that
\[\v(\ulam) = \v(0) - \lb\xi - \frac{\Dl}{\D}\teta\rb\ulam, \quad \mu_{0}(\ulam) = \mu_{0} + \frac{\Dl}{\D}\ulam.\]
Thus, searching for $\ulam$ is equivalent to solve simple linear equations. Algorithm \ref{algo:find_utilde_lambda}
\begin{algorithm}
  \caption{FIND\_UTILDE\_LAMBDA}\label{algo:find_utilde_lambda}
  \textbf{Input: } Support $\S$, iterate $\v = \com{\xs}{-\muc}$, intermediate variables $\Par_{1}, \Par_{3}$.

\textbf{Procedure: }
\begin{algorithmic}[1]
  \STATE $\ulam^{\mathrm{inc}}\gets \mathrm{min}_{+}\left\{\frac{\v_{i}}{\xi_{i} - \frac{\Dl}{\D}\teta_{i}}: i = 1, 2, \ldots, n\right\}$;
  \STATE $j \gets \mathrm{argmin}_{+}\left\{\frac{\v_{i}}{\xi_{i} - \frac{\Dl}{\D}\teta_{i}}: i = 1, 2, \ldots, n\right\}$;
  \IF {$j\in \S$}
  \STATE $\S^{\mathrm{new}} = \S \setminus \{j\}$;
  \ELSE
  \STATE $\S^{\mathrm{new}} = \S\cup \{j\}$.
  \ENDIF
\end{algorithmic}

\textbf{Output: }$\ulam^{\mathrm{inc}}, j, \S^{\mathrm{new}}$.
\end{algorithm}

\subsection{Variables Update}
\subsubsection{\textnormal{UPDATE\_BY\_UTILDE\_LAMBDA}}
Note that all intermediate variables are not affected by $\ulam$, we only need to update $v$ and $\mu_{0}$ accordingly.
\begin{algorithm}
\caption{UPDATE\_BY\_UTILDE\_LAMBDA}\label{algo:update_by_utilde_lambda}
\textbf{Input: } Increment $\ulam^{\mathrm{inc}}$; iterate $\v = \com{\xs}{-\muc}$, $\mu_{0}$; intermediate variables $\Par_{1}, \Par_{3}$.

\textbf{Procedure: }
\begin{algorithmic}[1]
  \STATE $\v \gets \v - \lb \xi - \frac{\Dl}{\D}\teta\rb \ulam^{\mathrm{inc}}$;
  \STATE $\mu_{0}\gets \mu_{0} + \frac{\Dl}{\D}\ulam^{\mathrm{inc}}$.
\end{algorithmic}

\textbf{Output: }$\v, \mu_{0}$.
\end{algorithm}

\subsubsection{\textnormal{UPDATE\_UTILDE\_EXPAND\_SUPPORT}}
Since $\M$ is exactly the same as in Appendix \ref{app:varyA}, we can directly apply Theorem \ref{thm:expand_M} to obtain an update of $\M$ and the updates of other parameters as a consequence.
\begin{theorem}\label{thm:expand_other_utilde}
Let $\gamma$ and $\td{\gamma}$ be defined in Theorem \ref{thm:expand_M}, i.e.
\[\gamma_{\T} = \td{\gamma}_{\T} = \rcom{\Mjs}{1}^{T}, \quad \gamma_{\T^{c}} = -\ACj - \ACs\Mjs^{T}, \quad \td{\gamma}_{\T^{c}} = 0,\]
then
  \begin{itemize}
  \item $\M^{+} = \opr(\M) + \frac{1}{\tAjj}\gamma\td{\gamma}^{T}$;
  \item $\teta^{+} = \opr(\teta) + \frac{\td{\eta}_{j}}{\tAjj}\gamma$;
  \item $\D^{+} = \D + \frac{\teta_{j}^{2}}{\tAjj}$;
  \item $\xi^{+} = \opr(\xi) + \frac{\xi_{j}}{\tAjj}\gamma$;
  \item $\Dl^{+} = \Dl + \frac{\xi_{j}\teta_{j}}{\tAjj}$.
  \end{itemize}
\end{theorem}
\begin{proof}
The update of $\M$, $\teta$ and $\D$has been proved in Theorem \ref{thm:expand_M}. For any subset $\S$, let $I_{\S}$ denote the matrix with $j$-th diagonal element equal to 1 for any $j\in \S$ and all other elements equal to 0. Then $\xi$ and $\xi^{+}$ can be rewritten as
\[\xi = -\lb \M + I_{\S^{c}}\rb\l, \quad \xi^{+} = -\lb \M^{+} + I_{\T^{c}}\rb\l.\]
Note that $I_{\S^{c}} - I_{\T^{c}} = e_{j}e_{j}^{T}$ where $e_{j}$ is the $j$-th basis vector, then we have
\begin{align*}
  \xi^{+} - \xi &= \lb \M - \M^{+} + e_{j}e_{j}^{T} \rb\l = \lb M - \opr(M) -\frac{1}{\tAjj}\gamma\td{\gamma}^{T} + e_{j}e_{j}^{T}\rb\l = -\frac{\td{\gamma}^{T}\l}{\tAjj}\gamma + (\l_{j} - \Mjs\ls)e_{j}\\
& \quad \Longrightarrow \xi^{+} = \xi - \frac{\td{\gamma}^{T}\l}{\tAjj}\gamma + (\l_{j} - \Mjs\ls)e_{j} = \opr(\xi) - \frac{\td{\gamma}^{T}\l}{\tAjj}\gamma.
\end{align*}
Note that $\td{\gamma}^{T}\l = \l_{j} + \Mjs\ls = -\xi_{j}$, we obtain that 
\[\xi^{+} = \opr(\xi) + \frac{\xi_{j}}{\tAjj}\gamma.\]
For $\Dl^{+}$, we have
\begin{align*}
  \Dl^{+} &= \onet^{T}\xit^{+} = \Dl + \frac{\xi_{j}}{\tAjj}\cdot \onet^{T}\gamma_{\T} = \Dl + \frac{\xi_{j}\teta_{j}}{\tAjj}.
\end{align*}
\end{proof}

The implementation of UPDATE\_TILDE\_EXPAND\_SUPPORT is summarized in Algorithm \ref{algo:update_utilde_expand_support}.
\begin{algorithm}
  \caption{UPDATE\_UTILDE\_EXPAND\_SUPPORT}\label{algo:update_utilde_expand_support}
  \textbf{Inputs: }Original support $S$, new index $j$, matrix $A$, vector $\l$, intermediate variables $\Par_{1}, \Par_{3}$.

\textbf{Procedure: }
\begin{algorithmic}[1]
  \STATE $\tAjj\gets \Ajj + \Mjs\Asj$;
  \STATE $\gamma_{\T}\gets (\Mjs, 1)^{T}, \gamma_{\T^{c}}\gets -\A_{\T^{c}j} - \A_{\T^{c}\S}\Mjs^{T}$;
  \STATE $\td{\gamma}_{\T}\gets (\Mjs, 1)^{T}, \td{\gamma}_{\T^{c}}\gets 0$;
  \STATE $\D\gets \D + \frac{\teta_{j}^{2}}{\tAjj}$;
  \STATE $\Dl\gets \Dl + \frac{\xi_{j}\teta_{j}}{\tAjj}$;
  \STATE $\xi\gets \opr(\xi) + \frac{\xi_{j}}{\tAjj}\gamma$;
  \STATE $\teta\gets \opr(\teta) + \frac{\teta_{j}}{\tAjj}\gamma$;
  \STATE $\Mt\gets \opr(\Mt) + \frac{1}{\tAjj}\gamma\td{\gamma}_{\T}^{T}$.
\end{algorithmic}

\textbf{Output: }$\Par_{1}, \Par_{3}$.
\end{algorithm}

\subsubsection{\textnormal{UPDATE\_UTILDE\_SHRINK\_SUPPORT}}
Since $\M$ is exactly the same as in Appendix \ref{app:varyA}, we can directly apply Theorem \ref{thm:shrink_M} to obtain an update of $\M$ and the updates of other parameters as a consequence.
\begin{theorem}\label{thm:expand_other_utilde}
Let $\beta$ and $\td{\beta}$ be defined in Theorem \ref{thm:shrink_M}, i.e. 
\[\beta_{\T} = \td{\beta}_{\T} = \Mtj, \quad \beta_{\T^{c}} = \com{-1}{\Mcj}\quad \td{\beta}_{\T^{c}} = 0,\]
then
  \begin{itemize}
  \item $\M^{-} = \opr(\M) - \frac{1}{\Mjj}\cdot\beta\td{\beta}^{T}$;
  \item $\teta^{-} = \opr(\teta) - \frac{\teta_{j}}{\Mjj}\beta$;
  \item $\D^{-} = \D - \frac{\teta_{j}^{2}}{\Mjj}$;
  \item $\xi^{-} = \opr(\xi) - \frac{\xi_{j}}{\Mjj}\beta$;
  \item $\Dl^{-} = \Dh - \frac{\xi_{j}\teta_{j}}{\Mjj}$.
  \end{itemize}
\end{theorem}
\begin{proof}
The update of $\M$, $\teta$ and $\D$ has been proved in Theorem \ref{thm:shrink_M} and Theorem \ref{thm:shrink_other}. For any subset $\S$, let $I_{\S}$ denote the matrix with $j$-th diagonal element equal to 1 for any $j\in \S$ and all other elements equal to 0. Then $\xi$ and $\xi^{-}$ can be rewritten as
\[\xi = -\lb \M + I_{\S^{c}}\rb\l, \quad \xi^{-} = -\lb \M^{-} + I_{\T^{c}}\rb\l.\]
Note that $I_{\T^{c}} - I_{\S^{c}} = e_{j}e_{j}^{T}$ where $e_{j}$ is the $j$-th basis vector, then we have
\begin{align*}
  \xi^{-} - \xi &= \lb \M - \M^{-} - e_{j}e_{j}^{T} \rb\l = \lb \frac{1}{\Mjj}\beta\td{\beta}^{T} + \M - \opr(M) - e_{j}e_{j}^{T}\rb\l\\ 
& = \frac{\beta_{\T}^{T}\l_{\T}}{\Mjj}\beta + \lb \M - \opr(M) - e_{j}e_{j}^{T}\rb\l \triangleq \frac{\beta_{\T}^{T}\l_{\T}}{\Mjj}\del + \td{\xi}.
\end{align*}
By definition of $\td{\xi}$
\[\td{\xi}_{\T} = \l_{j}\Mtj, \quad \td{\xi}_{j} = \Mjt\lt + \l_{j}\Mjj - \l_{j} = -\xi_{j} - \l_{j},\quad \td{\xi}_{\S^{c}} = \l_{j}\Mcj,\]
and thus, 
\[\td{\xi} = \l_{j}\beta -\xi_{j}e_{j}.\]
This implies that 
\[\xi^{-} = \xi -\xi_{j}e_{j} + \frac{\beta_{\T}^{T}\l_{\T} + \l_{j}\Mjj}{\Mjj}\beta = \opr(\xi) - \frac{\xi_{j}}{\Mjj}\beta.\]
For $\Dl^{-}$, we have
\begin{align*}
  \Dl^{-} &= \onet^{T}\xit^{-} = \Dl - \xi_{j} - \frac{\xi_{j}}{\Mjj}\cdot \onet^{T}\beta_{\T} = \Dl - \frac{\xi_{j}(\onet^{T}\beta_{\T} + \Mjj)}{\Mjj} = \Dl - \frac{\xi_{j}\teta_{j}}{\Mjj}.
\end{align*}
\end{proof}

The implementation of UPDATE\_TILDE\_SHRINK\_SUPPORT is summarized in Algorithm \ref{algo:update_utilde_shrink_support}.
\begin{algorithm}
  \caption{UPDATE\_UTILDE\_SHRINK\_SUPPORT}\label{algo:update_utilde_shrink_support}
  \textbf{Inputs: }Original support $S$, new index $j$, vector $\l$, intermediate variables $\Par_{1}, \Par_{3}$.

\textbf{Procedure: }
\begin{algorithmic}[1]
  \STATE $\beta_{\T}\gets \Mjt^{T}, \beta_{\T^{c}}\gets \com{-1}{\MCj}, \td{\beta}_{\T}\gets \Mjt^{T}, \td{\beta}_{\T^{c}}\gets 0$;
  \STATE $\D\gets \D - \frac{\teta_{j}^{2}}{\Mjj}$;
  \STATE $\Dl\gets \Dl - \frac{\xi_{j}\teta_{j}}{\Mjj}$;
  \STATE $\xi\gets \opr(\xi) - \frac{\xi_{j}}{\Mjj}\beta$;
  \STATE $\teta\gets \opr(\td{\eta}) - \frac{\teta_{j}}{\Mjj}\beta$;
  \STATE $\M_{\bigcdot, \T}\gets \opr(\M_{\bigcdot, \T}) - \frac{1}{\Mjj}\beta\td{\beta}_{\T}^{T}, \,\,\,\M_{\bigcdot, j}\gets 0$.
\end{algorithmic}

\textbf{Output: }$\Par_{1}, \Par_{3}$.
\end{algorithm}

\subsubsection{\textnormal{DIRECT\_UTILDE\_UPDATE}}
At the beginning of each time $t$, we need to recompute $\xi$ and $\Dl$. The implementation is summarized in Algorithm \ref{algo:direct_utilde_update}.
\begin{algorithm}
  \caption{DIRECT\_UTILDE\_UPDATE}\label{algo:direct_utilde_update}
  \textbf{Inputs: } Support $S$, vector-update-vector $\l$, intermediate variables $\M$.

\textbf{Procedure: }

\begin{algorithmic}[1]
  \STATE $\xis\gets -\Mss\ls, \,\,\xic \gets -\lc - \Mcs\ls$;
  \STATE $\Dl\gets \ones^{T}\xis$.
\end{algorithmic}

\textbf{Output: }$\xi, \Dl$.
\end{algorithm}

\subsection{Complexity Analysis}
Similar to Appendix \ref{app:varyA}, we can analyze the computation complexity. The analysis here is much simpler than the last case since the implementation is quite straightforward. Table \ref{tab:complexity_c} summarizes the results.
\begin{table}[h]
  \centering
  \caption{Computation complexity of each sub-routine in Algorithm \ref{algo:real_time_hop_ulam}.}\label{tab:complexity_c}
  \begin{tabular}{lllll}
    \toprule
    & ($Wz$)-type & ($z\td{z}^{T}$)-type & ($z^{T}z$)-type & ($az$)-type\\ 
    \midrule
    FIND\_UTILDE\_LAMBDA & 0 & 0 & $2n$ & 0\\
    UPDATE\_BY\_UTILDE\_LAMBDA & 0 & 0 & 0 & $n$\\
    UPDATE\_UTILDE\_EXPAND\_SUPPORT & $(n - s - 1)s$ & $n(s + 1)$ & $s$ & $2n$\\
    UPDATE\_UTILDE\_SHRINK\_SUPPORT & 0 & $n(s - 1)$ & $0$ & $2n$\\
    DIRECT\_UTILDE\_UPDATE & $ns$ & 0 & $s$ & 0\\
    \bottomrule
  \end{tabular}
\end{table}

Let $k_{\c}^{+}$ and $k_{\c}^{-}$ be the number of turning points that $\S$ is expanded and shrinked respectively and $k_{\c} = k_{\c}^{+} + k_{\c}^{-}$ be the total number of tuning points, then the complexity is 
\[C_{2t} \le ns + n + 3nk_{\c} + n(2s + 3)k_{\c}^{+} + n(s + 1)k_{\c}^{-} + O(k_{\c}) = ns(2k_{\c} + 1) + n(6k_{\c} + 1) + O(k_{\c}).\]

\section{Implementation of \OurMethod Algorithm With Time-Varying $\A, \c$}\label{app:varyAc}
\subsection{Intermediate Variables}
Based on the results in Appendix \ref{app:varyA} and Appendix \ref{app:varyc}, we can concatenate Algorithm \ref{algo:real_time_hop} and Algorithm \ref{algo:real_time_hop_ulam}. Thus we define $\Par_{1}, \Par_{2}, \Par_{3}$ as $\Par_{1} = \{\M, \teta, \D\}, \Par_{2} = \{\eta, \Dg, \Dgg, \Dgy\}, \Par_{3} = \{\xi, \Dl\}$ where all parameters are defined in previous appendices. 

\subsection{Implementation}
Note that only two sub-routines involves the matrix $\A$, namely UPDATE\_EXPAND\_SUPPORT and UPDATE\_UTILDE\_EXPAND\_SUPPORT,  and moreover they only involve the $j$-th column of $\A$. Thus, we can use the sparse update of $\A$ as in Algorithm \ref{algo:real_time_hop_smart_A} for acceleration. Algorithm \ref{algo:real_time_hop_varyAc} below describes the implementation.

\subsection{Complexity Analysis}
The complexity of Algorithm \ref{algo:real_time_hop_varyAc} is just the sum of that of Algorithm \ref{algo:real_time_hop} and Algorithm \ref{algo:real_time_hop_ulam}, i.e. 
\[C_{t} = C_{1t} + C_{2t} = ns_{*} + ns(3k_{\A} + 2k_{\c} + 2) + n(12k_{\A} + 6k_{\c} + 3) + O(k_{\A} + k_{\c}).\]

\begin{algorithm}
\caption{\OurMethod Algorithm for time-varying $\A, \c$ with sparse update of $\A$}\label{algo:real_time_hop_varyAc}
\textbf{Inputs: } Initial parameters $\A^{(0)}$, vectors $\{\c^{(t)}: t = 1, 2, \ldots\}$, 

\hspace{12mm} matrix-update-vectors $\{g^{(t)}, t = 1, 2, \ldots\}$.

\textbf{Initialization: }
\begin{algorithmic}
  \STATE $x\gets$ as the optimum corresponding to $A^{(0)}, \c^{(0)}$.
  \STATE $\S\gets \supp(x), \S_{*}\gets \S$;
  \STATE Calculate $(\x, \mu, \mu_{0})$ via (\ref{eq:xs})-(\ref{eq:muzero})
  \STATE $\v \gets (\xs, -\muc)$;
  \STATE Calculate intermediate variables $(\Par_{1}, \Par_{2})$ via (\ref{eq:param_mat})-(\ref{eq:param_scalar}) based on $\c^{(0)}, \g^{(1)}$;

\end{algorithmic}

\textbf{Procedure: }
\begin{algorithmic}[1]
  \FOR{$t = 1, 2, \cdots$.}
  \STATE $\lambda\gets 0$;
  \WHILE{$\lambda < 1$}
     \STATE $(\lam^{\mathrm{inc}}, j, S^{\mathrm{new}})\gets \mathrm{FIND\_LAMBDA}(S, \v; \Par_{1}, \Par_{2})$;
     \STATE $\lam^{\mathrm{inc}} \gets \min\{\lam^{\mathrm{inc}}, 1 - \lam\}$;
     \STATE $\lam\gets \lam + \lam^{\mathrm{inc}}$;
     \STATE $(\v, \mu_{0}; \Par_{1}, \Par_{2})\gets \mathrm{UPDATE\_BY\_LAMBDA}(\lam^{\mathrm{inc}}; \v, \mu_{0}, \Par_{1}, \Par_{2})$;
     \IF {$S^{\mathrm{new}} = S\cup \{j\}$}
        \STATE $(\Par_{1}, \Par_{2})\gets \mathrm{UPDATE\_EXPAND\_SUPPORT}(\lam, S, j; \A, \c^{(t-1)}, \g^{(t)}, \Par_{1}, \Par_{2})$;
        \IF {$j\not\in \S_{*}$}
        \STATE $G\gets (g^{(1)}, \ldots, g^{(t - 1)})$;
        \STATE $\A_{\bigcdot, j}\gets \A_{\bigcdot, j} + GG_{j, \bigcdot}^{T}$;
        \STATE $\S_{*} = \S_{*}\cup \{j\}$;
        \ENDIF
     \ELSIF {$S^{\mathrm{new}} = S\setminus \{j\}$}
        \STATE $(\Par_{1}, \Par_{2})\gets \mathrm{UPDATE\_SHRINK\_SUPPORT}(S, j; \c^{(t - 1)}, \g^{(t)}, \Par_{1}, \Par_{2})$;
     \ENDIF
     \STATE $S\gets S^{\mathrm{new}}$;
  \ENDWHILE
  \STATE $A_{\bigcdot, \S_{*}}\gets A_{\bigcdot, \S_{*}} + g^{(t)}(g_{\S_{*}}^{(t)})^{T}$;
  \STATE $\l^{(t)}\gets \c^{(t)} - \c^{(t - 1)}$;
  \STATE $\Par_{3}\gets \mathrm{DIRECT\_UTILDE\_UPDATE}(S, \Par_{1}, \l^{(t)})$;
  \STATE $\ulam\gets 0$;
  \WHILE{$\ulam < 1$}
  \STATE $(\ulam^{\mathrm{inc}}, j, S^{\mathrm{new}})\gets \mathrm{FIND\_UTILDE\_LAMBDA}(\v; \Par_{1}, \Par_{3})$;
  \STATE $\ulam^{\mathrm{inc}} \gets \min\{\ulam^{\mathrm{inc}}, 1 - \ulam\}$;
  \STATE $(\v, \mu_{0})\gets \mathrm{UPDATE\_BY\_UTILDE\_LAMBDA}(\ulam^{\mathrm{inc}}; \v, \mu_{0}, \Par_{1}, \Par_{3})$;
  \IF {$S^{\mathrm{new}} = S\cup \{j\}$}
  \STATE $(\Par_{1}, \Par_{3})\gets \mathrm{UPDATE\_UTILDE\_EXPAND\_SUPPORT}(S, j, \A, \l^{(t)}; \Par_{1}, \Par_{3})$;
        \IF {$j\not\in \S_{*}$}
        \STATE $G\gets (g^{(1)}, \l^{(t)}dots, g^{(t - 1)})$;
        \STATE $\A_{\bigcdot, j}\gets \A_{\bigcdot, j} + GG_{j, \bigcdot}^{T}$;
        \STATE $\S_{*} = \S_{*}\cup \{j\}$;
        \ENDIF
  \ELSIF {$S^{\mathrm{new}} = S\setminus \{j\}$}
  \STATE $(\Par_{1}, \Par_{3})\gets \mathrm{UPDATE\_UTILDE\_SHRINK\_SUPPORT}(S, j, \l^{(t)}; \Par_{1}, \Par_{3})$;
  \ENDIF
  \STATE $S\gets S^{\mathrm{new}}$;
  \STATE $\ulam\gets \ulam + \ulam^{\mathrm{inc}}$.
  \ENDWHILE
  \STATE $\Par_{2}\gets \mathrm{DIRECT\_UPDATE}(S, \c^{(t)}, g^{(t + 1)}; \Par_{1}, \Par_{2})$;
  \STATE $x_{\S}^{(t)}\gets \xs, \quad x_{\S^{c}}^{(t)} \gets 0$.
  \ENDFOR
\end{algorithmic}

\textbf{Output:} $x^{(1)}, x^{(2)}, \cdots$.
\end{algorithm}
\end{document}